\documentclass[review,onefignum,onetabnum]{siamart250211}

\usepackage{lipsum}
\usepackage{amsfonts}
\usepackage{graphicx}
\usepackage{epstopdf}
\usepackage{algorithmic}
\usepackage{amssymb}
\usepackage[english]{babel}
\usepackage{array}
\usepackage{adjustbox}
\usepackage{relsize}
\usepackage{spverbatim}
\usepackage{listings}
\usepackage{mathrsfs}
\usepackage{derivative}
\usepackage{mathtools}
\usepackage{float}
\usepackage[section]{placeins}
\usepackage[numbers,sort]{natbib}

\usepackage{hyperref}
\usepackage{cleveref}

\usepackage{etoolbox} 
\makeatletter
\newcommand{\crefcomma}[1]{%
  \begingroup
    \def\crefcomma@sep{}%
    \forcsvlist{\crefcomma@do}{#1}%
  \endgroup
}
\newcommand{\crefcomma@do}[1]{%
  \ifx\crefcomma@sep\@empty\else,~\fi
  \cref{#1}%
  \def\crefcomma@sep{,}%
}
\newcommand{\Crefcomma}[1]{%
  \begingroup
    \def\crefcomma@sep{}%
    \forcsvlist{\Crefcomma@do}{#1}%
  \endgroup
}
\newcommand{\Crefcomma@do}[1]{%
  \ifx\crefcomma@sep\@empty\else,~\fi
  \Cref{#1}%
  \def\crefcomma@sep{,}%
}
\makeatother

\DeclareMathAlphabet{\mathpzc}{OT1}{pzc}{m}{it}

\newcolumntype{L}{>{$}l<{$}}

\ifpdf
  \DeclareGraphicsExtensions{.eps,.pdf,.png,.jpg}
\else
  \DeclareGraphicsExtensions{.eps}
\fi

\newsiamremark{hypothesis}{Hypothesis}
\crefname{hypothesis}{Hypothesis}{Hypotheses}
\newsiamthm{claim}{Claim}
\newtheorem{remark}{Remark}
\nolinenumbers

\headers{Differential equations having coalescing turning points}{T. M. Dunster}

\title{Asymptotic expansions for solutions of differential equations having coalescing turning points, with an application to Legendre functions}

\author{T. M. Dunster\thanks{Department of Mathematics and Statistics, San Diego State University, 5500 Campanile Drive, San Diego, CA 92182-7720, USA. 
  (\email{mdunster@sdsu.edu}, \url{https://tmdunster.sdsu.edu}).}
  }

\usepackage{amsopn}

\makeatletter
\newcommand*{\addFileDependency}[1]{
  \typeout{(#1)}
  \@addtofilelist{#1}
  \IfFileExists{#1}{}{\typeout{No file #1.}}
}
\makeatother

\nolinenumbers
\ifpdf
\hypersetup{
  pdftitle={Asymptotic expansions for solutions of differential equations having a coalescing turning point and double pole, with an application to Legendre functions},
  pdfauthor={T. M. Dunster}
}
\fi
\begin{document}

\maketitle

\begin{abstract}
Linear second-order ordinary differential equations of the form $d^{2}w/dz^{2}=\{u^{2}f(a,z)$ $+g(z)\}w$ are studied for large values of the real parameter $u$, where $z$ ranges over a bounded or unbounded complex domain $Z$, and $a_{0} \le a \le a_{1} < \infty$. The functions $f(a,z)$ and $g(z)$ are analytic in the interior of $Z$. Moreover, $f(a,z)$ has exactly two real simple zeros in $Z$ for $a>a_{0}$ that depend continuously on $a$ and coalesce into a double zero as $a \to a_{0}$. Uniform asymptotic expansions are obtained for solutions in terms of parabolic cylinder functions and their derivatives, together with slowly varying coefficient functions. The coefficients are readily computable and explicit error bounds are provided. The results are then applied to derive new asymptotic expansions for the associated Legendre functions when both the degree $\nu$ and the order $\mu$ are large.
\end{abstract}



\section{Introduction}
\label{sec:Introduction}

We consider differential equations of the form
\begin{equation} 
\label{eq01}
\frac{d^{2}w}{dz^{2}}=\left\{u^{2}f(a,z)+g(z)\right\}w.
\end{equation}
For $a_{0}<a \leq a_{1} < \infty$ we assume that $f(a,z)$ has two real simple zeros at $z=z_{1}$ and $z=z_{2}$, which depend continuously on $a$, where $z_{1} < z_{2}$, and $z_{1} \to z_{2}$ as $a \to a_{0}$. We assume $f(a,z)$ and $g(z)$ are analytic in a certain domain $Z$ except possibly at the boundary of the domain. The domain $Z$ contains the two turning points, and $f(a,z)$ does not have other zeros in $Z$. In addition, $f(a,z)$ is assumed to be real for real $z=x \in Z \cap \mathbb{R}$, and $f(a,z)<0$ in $(z_{1},z_{2})$.  Both $f(a,z)$ and $g(z)$ may depend on $u$ since explicit error bounds are included in the Liouville-Green (LG) expansions that we employ.  Note that under the above hypotheses $f(a,z)>0$ for real $z \in Z$ such that $z<z_{1}$ and $z>z_{2}$.

These LG expansions arise from the standard transformation of equation (\ref{eq01}) (see \cite[Chap. 10]{Olver:1997:ASF}); here asymptotic expansions appear in the exponent \cite{Dunster:2020:LGE}, rather than multiplying an exponential \cite[Chap. 10]{Olver:1997:ASF}. Now, from \cite[Eq. (1.3)]{Dunster:2020:LGE}, define $\xi$ by
\begin{equation} 
\label{eq02}
\xi=\int_{z_{2}}^{z}\{f(a,t)\}^{1/2} dt.
\end{equation}
The branches are such that $\xi \geq 0$ for $z \geq z_{2}$ and by continuity elsewhere in $Z$ having a cut along $(-\infty,z_{2}]$. This transforms (\ref{eq01}) into the standard LG form 
\begin{equation} 
\label{eq03}
\frac{d^{2}W}{d \xi^{2}}
=\left\{u^{2}+\psi(a,z)\right\}W,
\end{equation}
where $W=f^{1/4}w$ and
\begin{equation}
\label{eq04}
\psi(a,z)=\frac{4f(a,z)f''(a,z)-5f'(a,z)^2}{16f(a,z)^3}
+\frac{g(z)}{f(a,z)}.
\end{equation}
It is clear that this is not analytic at the turning points $z=z_{j}$ ($j=1,2$) since $f(a,z)$ vanishes at these points, and this is why LG expansions are not valid there.

Fifty years ago, in a significant paper on the asymptotic theory of linear differential equations, Olver \cite{Olver:1975:SOL}  made a change of independent variable $\zeta$ and introduced a new parameter $\alpha$, to construct asymptotic solutions that are valid at the turning points. From Olver \cite[Eq. (2.8)]{Olver:1975:SOL} $\zeta$ is defined implicitly by
\begin{equation}
\label{eq05}
\frac{d\zeta}{dz}=\left(\frac{f(a,z)}{\zeta^{2}-\alpha^{2}}\right)^{1/2}.
\end{equation}
With $V=\{(\zeta^{2}-\alpha^{2})/f\}^{1/4}w$ (\ref{eq05}) transforms (\ref{eq01}) into the form 
\begin{equation} 
\label{eq06}
\frac{d^{2}V}{d \zeta^{2}}=\left\{u^{2}\left(\zeta^{2}-\alpha^2\right)
+\hat{\psi}(\alpha,\zeta)\right\}V,
\end{equation}
where
\begin{equation} 
\label{eq07}
\hat{\psi}(\alpha,\zeta)
=\frac{3\zeta^{2}+2\alpha^{2}}{4\left(\zeta^{2}-\alpha^{2}\right)^{2}}
+\left(\zeta^{2}-\alpha^{2}\right)\psi(a,z).
\end{equation}
Integration of (\ref{eq05}) and referring to (\ref{eq02}) yields the implicit equation for $\zeta$
\begin{multline} 
\label{eq08}
\xi= \int_{\alpha}^{\zeta}\left(t^2
-\alpha^{2}\right)^{1/2} dt
=\frac{1}{2}\zeta\left(\zeta^{2}-\alpha^{2}\right)^{1/2}
\\
-\frac{1}{2}\alpha^{2}\ln\left\{\zeta+\left(\zeta^{2}-\alpha^{2}\right)^{1/2}\right\}
+\frac{1}{2}\alpha^{2}\ln(\alpha).
\end{multline}
Note that as $\zeta \to \infty$
\begin{equation} 
\label{eq09}
\xi=\frac12\zeta^{2}-\frac{1}{4}\alpha^{2}
\left\{2\ln\left(\frac{2\zeta}{\alpha}\right)+1\right\}
+\mathcal{O}\left(\frac{1}{\zeta^{2}}\right).
\end{equation}

In \cref{eq02,eq08} the lower integration limits were chosen so that $\zeta=\alpha$ ($\xi=0$) corresponds to $z=z_{2}$. We also want $\zeta=-\alpha$ to correspond to $z=z_{1}$, and from 
(\ref{eq08}) this is achieved by defining $\alpha$ so that
\begin{equation} 
\label{eq10}
\int_{z_{1}}^{z_{2}}\{-f(a,t)\}^{1/2} dt
=\int_{-\alpha}^{\alpha}\left(\alpha^{2}-t^{2}\right)^{1/2} dt
=\frac12 \alpha^{2} \pi,
\end{equation}
and therefore
\begin{equation} 
\label{eq11}
\alpha = \left\{\frac{2}{\pi}
\int_{z_{1}}^{z_{2}}\{-f(a,t)\}^{1/2} dt\right\}^{1/2}.
\end{equation}
With this choice of $\alpha$, $\zeta$ is an analytic function of $z$ at $z=z_{j}$ ($j=1,2$). Moreover, unlike $\psi(a,z)$, one can verify that $\hat{\psi}(\alpha,\zeta)$ is also analytic at $\zeta=\pm \alpha$ (equivalently at $z=z_{1}$ and $z=z_{2}$). Note that $\alpha \to 0$ as $a \to a_{0}$, and so we let $0 \leq \alpha \leq \alpha_{1}$ correspond to $a_{0} \leq a \leq a_{1}$.

Olver \cite{Olver:1975:SOL} constructed solutions of (\ref{eq06}) of the form
\begin{equation} 
\label{eq12}
V\left(u,\alpha,\zeta\right)
=Y(u,\alpha,\zeta)
+\epsilon(u,\alpha,\zeta),
\end{equation}
where $Y(u,\alpha,\zeta)$ satisfies the so-called comparison equation
\begin{equation} 
\label{eq13}
\frac{d^{2}Y}{d \zeta^{2}}=u^{2}\left(\zeta^{2}
-\alpha^2\right)Y,
\end{equation}
which can be expressed in terms of parabolic cylinder functions. One such solution is $U(-\frac12 u\alpha^2,\sqrt{2u}\,\zeta)$, where $U(b,z)$ is a parabolic cylinder function which is a solution of Weber's equation
\begin{equation} 
\label{eq14}
\frac{d^{2}U}{dz^{2}}=\left(\frac14 z^2+b^2\right)U.
\end{equation}
It is an important solution of the equation since it is the unique one that is recessive at $z=+\infty$ (see \cite[Sect. 12.2]{NIST:DLMF}). In particular, we have
\begin{equation} 
\label{eq15}
U(b,z) \sim z^{-b-\frac12 }
e^{-\frac{1}4 z^{2}}
\quad \left(z \to \infty,\,
\vert \arg (z)\vert \le \tfrac{3}{4}\pi -\delta\right),
\end{equation}
and also
\begin{equation} 
\label{eq16}
U'(b,z) \sim -\tfrac12 z^{-b+\frac12 }
e^{-\frac{1}4 z^{2}}
\quad \left(z \to \infty,\, 
\vert \arg (z)\vert \le \tfrac{3}{4}\pi -\delta\right).
\end{equation}

Olver obtained explicit bounds for the error term $\epsilon$ which proved, assuming that $f$ and $g$ are bounded for large $u$ if dependent on that parameter, that $\epsilon=\mathcal{O}(u^{-p}\ln(u))$ as $u \to \infty$ relative to the magnitude of $Y$ except near the zeros of that function. Here $p=\tfrac23$ or $p=1$ depending on the sign of $f$ in the interval $(z_{1},z_{2})$. The bounds hold uniformly in an interval that is possibly unbounded or contains a singularity at one or both of its end points. He also obtained similar approximations and error bounds for the derivatives of these solutions. He also considered $u^{2}<0$ and $\alpha^{2}<0$. Our results can readily be extended to these cases, but we do not pursue those here.

Note that if the turning points are bounded away from one another, expansions at each turning point could be constructed in terms of Airy functions, using the well-known theory found in \cite[Chap. 11]{Olver:1997:ASF}. This results in simpler expansions, but the price paid is that the parameters are more restricted. The significance of Olver's paper is that his approximations are uniformly valid with respect to $\alpha \in [0,\alpha_{1}]$, and thus for the turning points bounded away from one another, as well as being allowed to coalesce into a double zero when $\alpha \to 0$ (equivalently $a \to a_{0}$). The latter situation translates into a larger permissible range for the second parameter. 

In a follow-up paper Olver applied the theory to Legendre functions in \cite{Olver:1975:LFW}, and later to Whittaker functions \cite{Olver:1980:WFW}. We too will apply the new expansions of the present paper to Legendre functions in \cref{sec:Legendre}, which extend Olver's one-term approximations valid for real argument to full expansions, and also that are valid for complex values of the argument. In our approach we do not consider the ODE (\ref{eq06}), rather we use the LG transformed equation \cref{eq03} only. However, the variable $\zeta$ given by \cref{eq02,eq08} still plays a central role in our analysis.

We obtain expansions involving $U(-\frac12 u\hat{\alpha}^2,\sqrt{2u}\,\zeta)$ and its derivative (along with other parabolic cylinder functions), where $\hat{\alpha}$ is a prescribed constant that is close in value to $\alpha$ for large $u$, but not necessarily exactly the same as that given by \cref{eq10}. The other parabolic cylinder functions that we shall use are $U(-\frac12 u\hat{\alpha}^2,-\sqrt{2u}\,\zeta)$ and $U(\frac12 u\hat{\alpha}^2,\pm i \sqrt{2u}\,\zeta)$. All four are solutions of (\ref{eq13}) with $\alpha$ replaced by $\hat{\alpha}$, and are significant because they are recessive at $\zeta = \pm \infty$ and $\zeta = \mp i\infty$, respectively.

We extend Olver's general results in several ways. Our results are valid for complex values of the argument rather than just real ones, with asymptotic expansions constructed rather than just the leading term. The coefficients in our new expansions are readily computable. We mention that Olver also showed how to obtain asymptotic expansions, but these involved very complicated expressions for the coefficients, were only shown to be valid in finite intervals, and are not valid at singularities. On the other hand, our expansions are valid in unbounded domains, as well as at certain singularities of the equation. Our results are rigorously justified with explicit and relatively simple error bounds. We also asymptotically solve the connection problem, which has applications to eigenvalue problems in quantum mechanics \cite{Barakat:2005:TAI}.

Earlier investigations into the problem of linear second-order differential equations having two turning points include those by Kazarinoff \cite{Kazarinoff:1958:ATO}, Langer \cite{Langer:1934:TAS,Langer:1959:TAS}, and McKelvey \cite{McKelvey:1955:TAS}. For other studies, see \cite{Heading:1962:PIM,Lynn:1970:UAS,Moriguchi:1959:AIO}. The results we develop here are significantly more general, and have applications to the asymptotic evaluation of various special functions, including Whittaker functions, prolate spheroidal wave functions, Bessel polynomials, and Laguerre polynomials. Physical applications include wave scattering problems for electric cylinders and prolate spheroids \cite{Kazarinoff:1958:ATO}, and the eigenfunctions of an anharmonic oscillator problem in quantum mechanics \cite{Ezawa:2012:CIM}.

The plan of this paper is as follows. In \cref{sec:LG} we construct LG solutions for \cref{eq01}, as well as establishing important connection formulas between them. These expansions are not themselves valid at $z=z_{1}$ and $z=z_{2}$, but will be used to construct asymptotic solutions that are valid at these turning points. In \cref{sec:PCF} we obtain LG expansions for the parabolic cylinder functions, and these, together with the results of \cref{sec:LG}, are used in \cref{sec:PCFExpansions} to obtain our main result, namely expansions that are valid at the turning points, uniformly for $a_{0} \leq a \leq a_{1}$. In \cref{sec:Legendre} we apply our new theory to obtain uniform asymptotic expansions for the associated Legendre functions when both the degree $\nu$ and the order $\mu$ are large. Finally in \cref{sec:numerics} we illustrate the accuracy of our uniform expansions with some numerical results from \cref{sec:Legendre} for two cases, where the turning points in question are both close and not close.

\section{LG expansions for the general equation}
\label{sec:LG}

Let points $z=x_{j}$ ($j=1,2,3,4$) in the $z$ plane correspond to $\zeta= i \infty$ ($\Re(\xi)=-\infty$), $\zeta=+\infty$ ($\Re(\xi)=+\infty$), $\zeta=-i \infty$ ($\Re(\xi)=-\infty$), and $\zeta =-\infty$ ($\Re(\xi)=+\infty$), respectively. Since $\zeta$ is unbounded at these points, the corresponding $z$ points must also be unbounded, or at a finite point that is a singularity of (\ref{eq01}). We assume that all four are boundary points of $Z$.

From \cref{eq02,eq08}, and recalling that $f(a,z)>0$ for $z<z_{1}$ and $z>z_{2}$, the points $x_{2}$ and $x_{4}$ lie on the real axis with $-\infty \leq x_{4} < z_{1} \leq z_{2} < x_{2} \leq \infty$, and $x_{1}$ and $x_{3}$ lie in the upper and lower half-planes, respectively.

At any finite point(s) $z=x_{j}$ ($j \in \{1,2,3,4\}$), which are singularities of \cref{eq01}, must have the properties:

(i) $f(a,z)$ has a pole of order $m>2$, and $g(z)$ is analytic or has a pole of order less than $\frac{1}{2}m+1$, or

(ii) $f(a,z)$ and $g(z)$ have a double pole, and 
$\left( z-x_{j}\right) ^{2}g(z) \to -\frac{1}{4}$ as $z\to x_{j}$.

On the other hand, if any of the point(s) $z=x_{j}$ ($j \in \{1,2,3,4\}$) are at infinity, then we assume that $f(a,z)$ and $g(z)$ can be expanded in convergent series in a neighborhood of $z=x_{j}$ of the forms
\begin{equation}
\label{eq17}
f(a,z)=z^{p}\sum\limits_{s=0}^{\infty }f_{s}(a)z^{-s},\ g(z)=z^{q}\sum\limits_{s=0}^{\infty }g_{s}{z}^{-s},
\end{equation}
where $f_{0}(a)\neq 0$, $g_{0}\neq 0$, and $p$ and $q$ are integers such that $p>-2$ and $q<\frac{1}{2}p-1$, or $p=q=-2$ and $g_{0}=-\frac{1}{4}$. The above requirements ensure that our LG expansions are uniformly valid as $z \to x_{j}$. For details and generalizations, see \cite[Chap. 10, \S\S 4 and 5]{Olver:1997:ASF}.

Let $Z_{j} \subset Z$ ($j=1,2,3$) comprise the point set for which there exists a path $\mathcal{L}_{j}$ in $Z$ linking $z$ with $x_{j}$ having the properties:

(i) $\mathcal{L}_{j}$ consists of a finite chain of $R_{2}$ arcs (as defined in \cite[Chap. 5, \S 3.3]{Olver:1997:ASF}); for example, a polygonal path.

(ii) $\Re(uv)$ is monotonic as $v$ passes along $\mathcal{L}_{j}$ from $x_{j}$ to $z$.

(iii) All points on the path $\mathcal{L}_{j}$ are bounded away from the turning points $z=z_{j}$ ($j=1,2$).

We shall call such paths $\mathcal{L}_{j}$ \emph{progressive}.

For each positive integer $n$ and all sufficiently large $u>0$, let $k_{n,j}(u,a)$, $j\in\{1,2,3\}$, be functions of $a$ that are continuous on $[a_0,a_1]$, chosen to satisfy certain conditions stated below. Then we apply \cite[Thm. 1.1]{Dunster:2020:LGE} to obtain the LG expansions to the solutions of (\ref{eq01}) that are recessive at $z=x_{j}$ ($j=1,2,3$), given by
\begin{equation}
\label{eq18}
w_{j}(u,a,z) 
=  \frac{k_{n,j}(u,a)W_{n,j}(u,a,z)}{f^{1/4}(a,z)}
\quad (j=1,2,3),
\end{equation}
where
\begin{equation}
\label{eq19}
W_{n,j}(u,a,z) = 
\exp \left\{u\xi+\sum\limits_{s=1}^{n-1}{
\frac{\hat{E}_{s}(a,z)}
{u^{s}}}\right\} \left\{1+\eta_{n,j}(u,a,z) \right\}
\quad (j=1,3),
\end{equation}
and
\begin{equation} 
\label{eq20}
W_{n,2}(u,a,z)  = 
\exp \left\{-u \xi+\sum\limits_{s=1}^{n-1}{(-1)^{s}
\frac{\hat{E}_{s}(a,z)}{u^{s}}}
\right\} \left\{1+\eta_{n,2}(u,a,z) \right\}.
\end{equation}
Here
\begin{equation} 
\label{eq21}
\hat{E}_{s}(a,z)=\int \hat{F}_{s}(a,z)f^{1/2}(a,z)dz 
\quad (s=1,2,3, \ldots ),
\end{equation}
where
\begin{equation} 
\label{eq22}
\hat{F}_{1}(a,z)=\frac12 \psi(a,z), \;
\hat{F}_{2}(a,z)=-\frac{1}{4f^{1/2}(a,z)}
\frac{d \psi(a,z)}{dz},
\end{equation}
with $\psi(a,z)$ given by (\ref{eq04}), and
\begin{multline} 
\label{eq23}
\hat{F}_{s+1}(a,z)=-\frac{1}{2f^{1/2}(a,z)}
\frac{d \hat{F}_{s}(a,z)}{dz}
-\frac{1}{2}\sum_{j=1}^{s-1}\hat{F}_{j}(a,z)
\hat{F}_{s-j}(a,z)
\\
\quad (s=2,3,4,\ldots).
\end{multline}
The integration constants in (\ref{eq21}) can be arbitrarily chosen, but for convenience, we assume that they are independent of $u$.

The coefficients $k_{n,j}(u,a)$ are chosen relative to certain connection formulas which we discuss later, and are such that each $w_{j}(u,a,z)$ ($j=1,2,3$) is independent of $n$ (and likewise for $w_{4}(u,a,z)$ given by \cref{eq28} below).

\begin{remark}
\label{evencoeffs}
As shown in \cite{Dunster:2017:COA} via Abel's identity, the even coefficients $\hat{E}_{2s}(a,z)$ satisfy the formal expansion
\begin{equation}
\label{even}
\sum\limits_{s=1}^{\infty }\frac{\hat{E}_{2s}(a,z) }{u^{2s}}
\sim -\frac{1}{2}\ln \left\{ 1+\sum\limits_{s=1}^{\infty}
\frac{\hat{F}_{2s-1}(a,z) }{u^{2s}}\right\} +\sum\limits_{s=1}^{\infty }\frac{h_{2s}(a)}{u^{2s}}
\quad (u \to \infty),
\end{equation}
where each $h_{2s}(a)$ can be arbitrarily chosen, and hence each $\hat{E}_{2s}(a,z)$ can be determined by expanding the RHS of \cref{even} in inverse powers of $u^{2}$, and then equating like powers. For example, $\hat{E}_{2}(a,z)=-\frac{1}{2}\hat{F}_{1}(a,z)+h_{2}(a)$. Thus, integration via (\ref{eq21}) is not required for these even coefficients.
\end{remark}

From (\ref{eq02}) and \cite[Thm. 1.1]{Dunster:2020:LGE} the error terms $\eta_{n,j}(u,a,z)$ ($j=1,2,3$) are $\mathcal{O}(u^{-n})$  as $u \to \infty$ uniformly for $z \in Z_{j}$, and satisfy the bounds in those regions
\begin{equation} 
\label{eq24}
\left\vert \eta_{n,j}(u,a,z) \right\vert 
\leq |u|^{-n}
\chi_{n,j}(u,a,z) \exp \left\{|u|^{-1}
T_{n,j}(u,a,z) +|u|^{-n}\chi_{n,j}(u,a,z) \right\},
\end{equation}
where
\begin{multline} 
\label{eq25}
\chi_{n,j}(u,a,z) 
=2\int_{x_{j}}^{z}
{\left\vert \hat{F}_{n}(a,v) \{f(a,v)\}^{1/2}dv\right\vert } 
\\ 
+\sum\limits_{s=1}^{n-1}\frac{1}{|u|^{s}}
 \sum\limits_{k=s}^{n-1}
\int_{x_{j}}^{z} {\left\vert 
\hat{F}_{k}(a,v) \hat{F}_{s+n-k-1}(a,v)
\{f(a,v)\}^{1/2}dv \right\vert },
\end{multline}
and
\begin{equation} 
\label{eq26}
T_{n,j}(u,a,z)
=4\sum\limits_{s=0}^{n-2} \frac{1}{|u|^{s}}
\int_{x_{j}}^{z}
\left\vert \hat{F}_{s+1}(a,v) \{f(a,v)\}^{1/2}dv \right\vert .
\end{equation}
The paths of integration are taken along $\mathcal{L}_{j}$, as described above. 

Next, consider the solution that is recessive at $z=x_{4}$. Since $\xi(z)$ has a branch cut along the real axis $(x_{4},z_{2}]$ we must consider two forms of the LG expansion, one which is valid above the cut, and the other below the cut. Now from (\ref{eq02}), (\ref{eq11}) and our assumptions on $f(a,z)$ we see that $\xi \to \infty \mp \frac12 \alpha^{2} \pi i$ as $z \to x_{4} \pm i0$. Now label the corresponding LG solutions by $W_{n,4}^{\pm}(u,a,z)$, respectively, and then these take the form
\begin{equation} 
\label{eq27}
W_{n,4}^{\pm}(u,a,z)  = 
\exp \left\{-u \xi+\sum\limits_{s=1}^{n-1}{(-1)^{s}
\frac{\hat{E}_{s}(a,z)}{u^{s}}}
\right\} \left\{1+\eta_{n,4}^{\pm}(u,a,z) \right\}.
\end{equation}
The only difference lies in the error terms, and these also satisfy \cref{eq24,eq25,eq26}, and where the integration paths $\mathcal{L}_{4}^{\pm}$, say, link $z$ to $x_{4} \pm i0$, while satisfying similar monotonicity conditions as the corresponding paths $\mathcal{L}_{j}^{\pm}$ ($j=1,2,3$) of the other three solutions. The two LG solutions given by (\ref{eq27}) are linearly dependent, due to their unique recessive behavior at $z=x_{4}$, but their error terms are bounded in different regions, which we label $Z_{4}^{\pm}$. 

Thus, we have as our fourth fundamental solution to (\ref{eq01}), recessive at $z=x_{4}$, given by the two forms
\begin{equation}
\label{eq28}
w_{4}(u,a,z) = 
\frac{k_{n,4}^{\pm}(u,a) W_{n,4}^{\pm}(u,a,z)}
{f^{1/4}(a,z)},
\end{equation}

Let us now examine the connection formulas between the four fundamental solutions. From linear second-order differential equation theory, we know that any one of them can be expressed as a linear combination of two of the others that are linearly independent. With this in mind, let the constants $k_{n,j}(u,a)$ ($j=1,2,3$) of \cref{eq18,eq28} be scaled relative to each other so that the linear relationship between $w_{j}(u,a,z)$ ($j=1,2,3$) is given by
\begin{equation}
\label{eq29}
w_{2}(u,a,z)
=w_{1}(u,a,z)+w_{3}(u,a,z),
\end{equation}
noting that $w_{1}(u,a,z)$ and $w_{3}(u,a,z)$ are linearly independent for our parameter range.

Given these values, we then express the linear relation between $w_{j}(u,a,z)$ ($j=1,3,4$) in the form
\begin{equation}
\label{eq30}
w_{4}(u,a,z)
=e^{\lambda \pi i}w_{1}(u,a,z)
+e^{-\lambda \pi i} w_{3}(u,a,z),
\end{equation}
for some $\lambda$, which is not necessarily real. Thus, from these two relations
\begin{equation}
\label{eq31}
2i\sin(\lambda\pi) w_{1}(u,a,z)
=w_{4}(u,a,z)
-e^{-\lambda \pi i} w_{2}(u,a,z),
\end{equation}
and
\begin{equation}
\label{eq32}
2i\sin(\lambda\pi)w_{3}(u,a,z)
=e^{\lambda \pi i} w_{2}(u,a,z)-w_{4}(u,a,z).
\end{equation}

As $z \to x_{1}$ we have $w_{1}(u,a,z) \to 0$ and $w_{j}(u,a,z)$ ($j=2,4$) unbounded. Now, the LG expansions \cref{eq19,eq27} are both asymptotically valid at $z = x_{1}$, provided that in the latter upper signs are taken. Thus, with these expansions, along with (\ref{eq31}), we deduce that
\begin{equation}
\label{eq33}
k_{n,4}^{+}(u,a)
=e^{-\lambda \pi i} k_{n,2}(u,a)
\left\{1+\mathcal{O}\left(u^{-n}\right)\right\}
\quad (u \to \infty).
\end{equation}
Similarly, on letting $z \to x_{3}$ in (\ref{eq32}),
\begin{equation}
\label{eq34}
k_{n,4}^{-}(u,a)
=e^{\lambda \pi i} k_{n,2}(u,a)
\left\{1+\mathcal{O}\left(u^{-n}\right)\right\}
\quad (u \to \infty).
\end{equation}

Next, from (\ref{eq29}) and letting $z \to x_{2}$
\begin{equation}
\label{eq35}
k_{n,1}(u,a)
=-k_{n,3}(u,a)
\left\{1+\mathcal{O}\left(u^{-n}\right)\right\}
\quad (u \to \infty).
\end{equation}

The following theorem provides an important asymptotic expansion for the connection coefficient $\lambda$ that appears in \cref{eq30}.
\begin{theorem}
\label{thm:lambda}
As $u \to \infty$
\begin{equation}
\label{eq36}
\lambda + \frac12
\sim \frac12 u \alpha^{2} 
+\frac{1}{2\pi i}\sum\limits_{s=0}^{\infty}{
\frac{\hat{E}_{2s+1}\left(a,x_{4} - i0\right)
-\hat{E}_{2s+1}\left(a,x_{4} + i0\right)}{u^{2s+1}}},
\end{equation}
where $\alpha$ is given by \cref{eq11}.
\end{theorem}

\begin{proof}
From (\ref{eq30}) and letting $z \to x_{4}$
\begin{equation}
\label{eq37}
e^{2\lambda \pi i}
=-\lim_{z \to x_{4}}\frac{w_{3}(u,a,z)}
{w_{1}(u,a,z)}.
\end{equation}
In this we use the LG expansions (\ref{eq19}) and take $z \to x_{4} \pm i0$ ($\xi \to \infty \mp \frac12 \alpha^{2} \pi i$) for $w_{1}(u,a,z)$ and $w_{3}(u,a,z)$, respectively. Then on referring to (\ref{eq35}), for any positive integer $n$, we arrive at
\begin{multline}
\label{eq38}
e^{2\lambda \pi i}
=-\exp \left\{u \alpha^{2} \pi i
+\sum\limits_{s=1}^{n-1}{
\frac{\hat{E}_{s}\left(a,x_{4} - i0\right)
-\hat{E}_{s}\left(a,x_{4} + i0\right)}
{u^{s}}}\right\} 
\\ \times
\left\{1+\mathcal{O}\left(\frac{1}{u^{n}}\right)\right\}
\quad (u \to \infty),
\end{multline}
on noting that
\begin{equation}
\label{eq39}
\frac{\lim_{z \to x_{4}-i0}f^{-1/4}(u,z)}
{\lim_{z \to x_{4}+i0}f^{-1/4}(u,z)}=-1,
\end{equation}
which follows from $f^{-1/4}(u,z)=\{(z-z_{1})(z-z_{2})\}^{-1/4}p(a,z)$, where $p(a,z)$ is analytic in $Z$. Now, from \cref{eq04,eq22,eq23,even}, one can show by induction that the even coefficients $\hat{E}_{2s}(a,z)$ ($s=1,2,3,\dots$) are meromorphic at the turning points $z= z_{j}$ ($j=1,2$), and so $\hat{E}_{2s}(a,x_{4} + i0)=\hat{E}_{2s}(a,x_{4} - i0)$. Thus (\ref{eq38}) yields \cref{eq36} (modulo an arbitrary integer, which we take to be zero)
\end{proof}

We finish with a result that will be referred to later.
\begin{lemma}
\label{lem:ratio}
For arbitrary $\delta>0$ assume that $a_{0}+\delta \leq a \leq a_{1}$. Then for any positive integer $n$
\begin{equation}
\label{eq40}
\frac{k_{n,1}(u,a)}{k_{n,2}(u,a)}
= i \exp\left\{\sum_{s=0}^{n-1}
\frac{\kappa_{2s+1}(a)}{u^{2s+1}} \right\}
\left\{1+ \mathcal{O}
\left(\frac{1}{u^{n}}\right)\right\}
\quad (u \to \infty),
\end{equation}
for certain constants $\kappa_{2s+1}(a)$ which depend on the odd coefficients in (\ref{eq21}).
\end{lemma}

\begin{proof}

We use (\ref{eq29}), and in this relation let $z \to x_{3}$, in which case $w_{3}(u,a,z) \to 0$, while the other two functions grow exponentially, but whose ratio is bounded in this limit. We use the LG expansions \cref{eq18,eq19,eq20} for $w_{1}(u,a,z)$ and $w_{2}(u,a,z)$ for an arbitrary positive integer $n$. To this end, observe from \cref{eq09} that $\xi \to -\infty + \frac14 \alpha^{2} \pi i$ as $z \to x_{3}$ ($\zeta \to -i \infty$). This can be used for the LG expansion of $w_{2}(u,a,z)$, since $x_{3}$ lies in the LG region of validity, $Z_{2}$, for $W_{n,2}(u,z)$.

The determination of the corresponding LG limit as $z \to x_{3}$ for $w_{1}(u,a,z)$ is not quite as straightforward, since the variable must be accessed across the cut $[z_{1},z_{2}]$ in the $z$ plane, as this is the only way to enter the lower half-plane along a progressive path $\mathcal{L}_{1}$ used in the error bound. This path can then be chosen to satisfy the condition that $\Re(u\xi(v))$ be monotonic for points $v$ moving along the path from the end points $x_{1}$ to $x_{3}$. In addition, the path must be bounded away from both turning points as it passes through $[z_{1},z_{2}]$, and that is why we require them to be bounded away from each other, given by the condition $a_{0}< a_{0}+\delta \leq a$.

Consider then $z \to x_{3}$ from crossing the part of the cut $[z_{1},z_{2}]$ from above, and label this $z=x_{3}^{-}$. Then from \cref{eq09} we find $\xi \to +\infty - \frac14 \alpha^{2} \pi i $ as $z \to x_{3}^{-}$ on this sheet. Thus, on using the fact that $w_{3}(u,a,z)/w_{2}(u,a,z) \to 0$ as $z \to x_{3}$, we have from \cref{eq18,eq29}
\begin{equation}
\label{eqA1}
1=\lim_{z \to x_{3}}\frac{w_{1}(u,a,z)}
{w_{2}(u,a,z)}
=\frac{k_{n,1}(u,a)\lim_{z \to x_{3}^{-}}
f^{-1/4}(u,z)W_{n,1}(u,a,z)}
{k_{n,2}(u,a)\lim_{z \to x_{3}}f^{-1/4}(u,z)
W_{n,2}(u,a,z)}.
\end{equation}

Next, similarly to \cref{eq39}, we have
\begin{equation}
\label{eqA2}
\frac{\lim_{z \to x_{3}^{-}}f^{-1/4}(u,z)}
{\lim_{z \to x_{3}}f^{-1/4}(u,z)}=-i.
\end{equation}
Further, since $\xi \to \pm \infty \mp \frac14 \alpha^{2} \pi i $ as $z \to x_{3}^{-}$ and $z \to x_{3}$, respectively, from \cref{eq19,eq20} it is seen that
\begin{multline} 
\label{eqA3}
\frac{\lim_{z \to x_{3}^{-}}W_{n,1}(u,a,z)}
{\lim_{z \to x_{3}}W_{n,2}(u,a,z)}
=\exp \left\{\sum\limits_{s=1}^{n-1}{
\frac{\hat{E}_{s}(a,x_{3}^{-})+(-1)^{s-1}\hat{E}_{s}(a,x_{3})}{u^{s}}}\right\} 
\\ \times
\left\{1+\mathcal{O}
\left(\frac{1}{u^{n}}\right) \right\}
\quad (u \to \infty).
\end{multline}
Since the even coefficients $\hat{E}_{2s}(a,z)$ ($s=1,2,3,\dots$) are meromorphic at the turning points $z= z_{j}$ ($j=1,2$), and therefore single-valued at $z=x_{3}$, it follows that $\hat{E}_{2s}(a,x_{3}^{-})-\hat{E}_{2s}(a,x_{3})=0$. Thus, \cref{eqA1,eqA2,eqA3} yields \cref{eq40}, where $\kappa_{2s+1}(a)=-\{\hat{E}_{2s+1}(a,x_{3}^{-})+\hat{E}_{2s+1}(a,x_{3})\}$.
\end{proof}

\section{LG expansions for parabolic cylinder functions}
\label{sec:PCF}

Here we present relevant results for parabolic cylinder functions, and then derive LG expansions for the ones that we shall use in our uniform asymptotic expansions for solutions of (\ref{eq01}). A comprehensive asymptotic study of these functions was given by \cite{Dunster:2021:UAP} for the case where the parameter is large. These involved LG and Airy function expansions, which in combination are valid for all values of the complex argument, including those at turning points. Here we consider the case where the parameter of the parabolic cylinder functions, namely $\pm \frac12 u \hat{\alpha}^{2}$, can be both small and large in absolute value. We do not require, or obtain, expansions that are valid at the turning points of the equation they satisfy.

Firstly, the following Wronskians hold (see \cite[Eqs. 12.2.10, 12.2.11 and 12.2.12]{NIST:DLMF}):
\begin{equation} 
\label{eq41}
\mathscr{W}\{U(-b,z),U(-b,-z)\}
=\frac{\sqrt{2\pi}}
{\Gamma\left(\tfrac{1}{2}-b\right)},
\end{equation}
\begin{equation}
\label{eq42}
\mathscr{W}\{U(b,z),V(b,z)\}
=\sqrt{2 /\pi},
\end{equation}
\begin{equation}
\label{eq43}
\mathscr{W}\{U(-b,z),U(b,\pm iz)\}
=\mp i e^{\mp \frac{1}{4}(2b-1)\pi i},
\end{equation}
and
\begin{equation} 
\label{eq44}
\mathscr{W}\{U(-b,-z),U(b,\pm iz)\}
=\pm i e^{\pm \frac{1}{4}(2b-1)\pi i}.
\end{equation}

Next, connection formulas are given as follows (\cite[Eqs. 12.2.13, 12.2.18, 12.2.19 and 12.2.20]{NIST:DLMF})
\begin{equation}
\label{eq45}
\sqrt{2\pi}U(-b,\pm z)
=\Gamma\left(b+\tfrac{1}{2}\right)
\left\{e^{\frac14(2b-1)\pi i}U(b,\pm iz) 
+ e^{-\frac14(2b-1)\pi i}
U\left(b, \mp iz\right)\right\},
\end{equation}
\begin{equation} 
\label{eq46}
U(-b,-z)=\mp ie^{\pm i\pi b}U(-b,z)
+\frac{\sqrt{2\pi}}
{\Gamma\left(\tfrac{1}{2}-b\right)}
e^{\pm \frac14(2b+1)\pi i}
U\left(b,\pm iz\right),
\end{equation}
and in the special case where the $U$ function is recessive at both $z=\pm \infty$
\begin{equation} 
\label{eq47}
U\left(-n-\tfrac12,-z\right)=(-1)^{n}U\left(-n-\tfrac12,z\right)
\quad (n=0,1,2,\ldots).
\end{equation}
In addition, we have the solution given by
\begin{equation} 
\label{eq48}
V(-b,z)=\frac{1}{\sqrt{2\pi}}
\left\{e^{\frac{1}{4}(2b+1)\pi i}U(b,iz)
+e^{-\frac{1}{4}(2b+1)\pi i}U(b,-iz)\right\}.
\end{equation}
We note that $U(-b,\pm z)$ and $V(-b,\pm z)$ are real-valued numerically satisfactory solutions of (\ref{eq14}) for $z \geq 0$ and $z \leq 0$, respectively.

Now consider the differential equation (\ref{eq13}), with $\alpha$ replaced by $\hat{\alpha}$ and $Y$ replaced by $\hat{Y}$,  namely
\begin{equation} 
\label{eq48a}
\frac{d^{2}\hat{Y}}{d \zeta^{2}}=u^{2}\left(\zeta^{2}
-\hat{\alpha}^2\right)\hat{Y}.
\end{equation}
At this stage, $\hat{\alpha}$ is regarded as an arbitrary nonnegative bounded parameter. As we mentioned in \cref{sec:Introduction}, this has solutions $U(-\frac12 u\hat{\alpha}^2,\pm \sqrt{2u}\,\zeta)$ and $U(\frac12 u\hat{\alpha}^2,\pm i \sqrt{2u}\,\zeta)$. 

Then, similarly to (\ref{eq02}), using \cite[Eq. (1.3)]{Dunster:2020:LGE} with $z=\zeta$ and $f=\zeta^2-\hat{\alpha}^{2}$, we introduce the LG variable given by
\begin{multline} 
\label{eq49}
\hat{\xi}
= \int_{\hat{\alpha}}^{\zeta}\left(t^2
-\hat{\alpha}^{2}\right)^{1/2} dt
=\frac{1}{2}\zeta\left(\zeta^{2}-\hat{\alpha}^{2}\right)^{1/2}
\\
-\frac{1}{2}\hat{\alpha}^{2}\ln\left\{\zeta+\left(\zeta^{2}-\hat{\alpha}^{2}\right)^{1/2}\right\}
+\frac{1}{2}\hat{\alpha}^{2}\ln(\hat{\alpha}).
\end{multline}
As $\zeta \to \infty$ we find that $\hat{\xi} \to \infty$, such that

\begin{figure}
 \centering
 \includegraphics[
 width=1.0\textwidth,keepaspectratio]{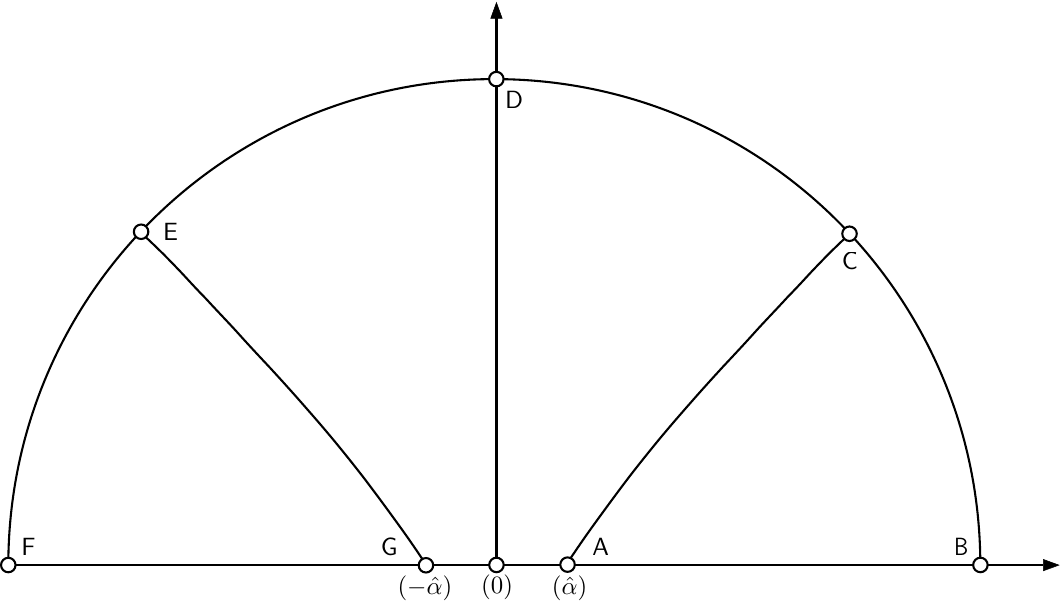}
 \caption{$\zeta$ plane.}
 \label{fig:zeta}
\end{figure}

\begin{figure}
 \centering
 \includegraphics[
 width=1.0\textwidth,keepaspectratio]{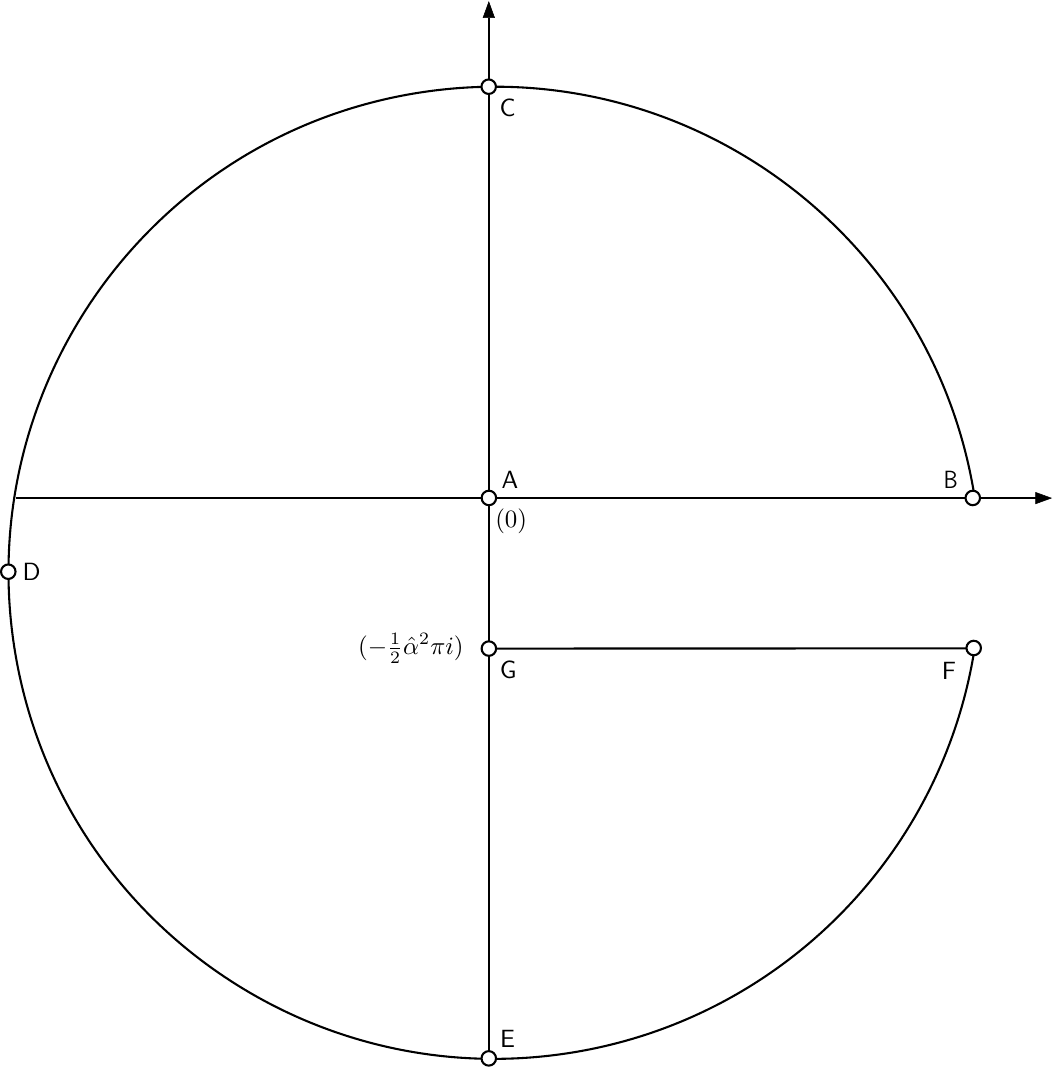}
 \caption{$\hat{\xi}$ plane.}
 \label{fig:xi}
\end{figure}

\begin{equation} 
\label{eq50}
\hat{\xi}=\frac12\zeta^{2}-\frac{1}{4}\hat{\alpha}^{2}
\left\{2\ln\left(\frac{2\zeta}{\hat{\alpha}}\right)+1\right\}
+\mathcal{O}\left(\frac{1}{\zeta^{2}}\right).
\end{equation}

The map of the upper $\zeta$ half plane to the $\hat{\xi}$ plane is shown in \cref{fig:zeta,fig:xi} with certain corresponding points labeled. The map of the lower half-plane is the conjugate of this. Thus, for example, at the points labeled $\mathsf{F}$, we observe $\hat{\xi} \to \infty -\frac12 \hat{\alpha}^2 \pi i$ as $\zeta \to -\infty + i0$ (see also (\ref{eq50})). Similarly $\hat{\xi} \to \infty +\frac12 \hat{\alpha}^2 \pi i$ as $\zeta \to -\infty - i0$ in the conjugate regions.

In \cref{sec:PCFExpansions}, our subsequent application to solutions of (\ref{eq01}) at the turning points of that equation, we shall take $\zeta$ to be defined by (\ref{eq08}), which will depend on $z$ and $\alpha$ (as given by \cref{eq11}). However, in what follows in this section, we consider it to be a general independent variable.

Now, let $\Theta_{1}$ be the upper half of the $\zeta$ plane ($\Im(\zeta) \geq 0$) with small neighborhoods of the turning points $\zeta = \pm \hat{\alpha}$ excluded, and $\Theta_{3}$ the conjugate region. Let $\Theta_{2}$ be the region in the $\zeta$ plane where $|\arg(\zeta)| \leq \frac34 \pi$ with a small neighborhood of the interval $[0,\hat{\alpha}]$ excluded. Let $\Theta_{4}^{+}$ be the region in the $\zeta$ plane where $\frac14 \pi \leq \arg(\zeta) \leq  \pi$ with a small neighborhood of the interval $[-\hat{\alpha},0]$ excluded, and finally let $\Theta_{4}^{-}$ be the conjugate of this. These are regions in which our respective LG solutions of \cref{eq48a} will be valid, and they all exclude the turning points as required. They are not maximal regions of asymptotic validity, but rather ones that are sufficiently large for our purposes.

On examination of \cref{fig:zeta,fig:xi} we see that the $\hat{\xi}$ region corresponding to $\Theta_{1}$ consists of all points shown in the illustrated part of the $\hat{\xi}$ plane, except for points at or near $\hat{\xi}=0$ and $\hat{\xi}=-\frac12 \hat{\alpha}^{2} \pi i$ (corresponding to the turning points $\zeta = \hat{\alpha}$ and $\zeta = -\hat{\alpha}+i0$). All points in this region can be accessed by a polygonal path linking them to $\hat{\xi} = - \infty$ (corresponding to $\zeta= i \infty$) in which $\Re(u\hat{\xi})$ is monotonic. Similarly, the $\hat{\xi}$ map of $\Theta_{3}$ consists of all points that can be accessed by a polygonal path linking them to $\hat{\xi} = - \infty$ (corresponding to $\zeta= -i \infty$) in which $\Re(u\hat{\xi})$ is again monotonic.

Likewise, the $\hat{\xi}$ region corresponding to $\Theta_{2}$ consists of points that can be accessed by a polygonal path linking them to $\hat{\xi} = +\infty$ (corresponding to $\zeta= +\infty$) in which $\Re(u\hat{\xi})$ is monotonic. Finally, $\hat{\xi}$ regions corresponding to $\Theta_{4}^{\pm}$ consist of points that can be accessed by a polygonal path that links them to $\hat{\xi} = \infty \mp \frac12 \hat{\alpha}^2 \pi i$ (corresponding to $\zeta= -\infty \pm i0$) in which $\Re(u\hat{\xi})$ is monotonic.

Now, from \cite[Eq. (1.6)]{Dunster:2020:LGE} with $z=\zeta$, $f=\zeta^2-\hat{\alpha}^{2}$ and $g=0$, one finds that the Schwarzian, which we denote by $\hat{\Psi}(\hat{\alpha},\zeta)$, is given by
\begin{equation} 
\label{eq51}
\hat{\Psi}(\hat{\alpha},\zeta)
=-\frac{3\zeta^{2}+2\hat{\alpha}^{2}}
{4\left(\zeta^{2}-\hat{\alpha}^{2}\right)^{3}}.
\end{equation}

Next, in order to evaluate the coefficients given by \cite[Eqs. (1.9) - (1.12)]{Dunster:2020:LGE}, we find it helpful to introduce the variable
\begin{equation} 
\label{eq52}
\hat{\beta}=\frac{\zeta}{\hat{\alpha}^{2}\left(\zeta^{2}-\hat{\alpha}^{2}\right)^{1/2}}-\frac{1}{\hat{\alpha}^{2}}
\quad (\hat{\alpha}>0),
\end{equation}
where the branch of the square root is positive for $\zeta>\hat{\alpha}$ and is continuous in the plane having a cut along $[-\hat{\alpha},\hat{\alpha}]$. In the limit $\hat{\alpha} \to 0$
\begin{equation} 
\label{eq53}
\hat{\beta}=\frac{1}{2\zeta^{2}}
\quad (\hat{\alpha}=0).
\end{equation}

We have, from \cref{eq52},
\begin{equation} 
\label{eq54}
\zeta=\frac{\hat{\alpha}^{2}\hat{\beta}+1}
{\left\{\hat{\beta}\left(\hat{\alpha}^{2}
\hat{\beta}+2\right)\right\}^{1/2}}.
\end{equation}
Now from (\ref{eq49}) it follows that $d\hat{\xi}/d\zeta=(\zeta^2-\hat{\alpha}^{2})^{1/2}$, and hence
\begin{equation} 
\label{eq55}
\frac{d\hat{\beta}}{d\hat{\xi}}=
\frac{d\hat{\beta}}{d\zeta}\left(
\frac{d\hat{\xi}}{d\zeta}\right)^{-1}
=-\hat{\beta}^{2}\left(\hat{\alpha}^{2}
\hat{\beta}+2\right)^{2}.
\end{equation}
Thus, for the coefficients \cite[Eq. (1.9)]{Dunster:2020:LGE}, let us denote $E_{s}(\xi)=e_{s}(\hat{\alpha},\hat{\beta})$. Then we 
find, using \cref{eq51,eq52,eq54,eq55}, and \cite[Eqs. (1.10) - (1.12)]{Dunster:2020:LGE}, that
\begin{equation} 
\label{eq56}
e_{1}(\hat{\alpha},\hat{\beta})=\tfrac{1}{24}\hat{\beta}
\left(5\hat{\alpha}^{4}\hat{\beta}^{2}
+15\hat{\alpha}^{2}\hat{\beta}+9\right),
\end{equation}
\begin{equation} 
\label{eq57}
e_{2}(\hat{\alpha},\hat{\beta})=\tfrac{1}{16}\hat{\beta}^{2}
\left(\hat{\alpha}^{2}\hat{\beta}+2\right)^{2}
\left(5\hat{\alpha}^{4}\hat{\beta}^{2}
+10\hat{\alpha}^{2}\hat{\beta}+3\right),
\end{equation}
and for $s=2,3,4\ldots$
\begin{multline} 
\label{eq58}
e_{s+1}(\hat{\alpha},\hat{\beta}) =
\frac{1}{2}\hat{\beta}^{2}\left(\hat{\alpha}^{2}
\hat{\beta}+2\right)^{2}
\frac{\partial e_{s}(\hat{\alpha},\hat{\beta})}
{\partial \hat{\beta}}
\\
+\frac{1}{2}\int_{0}^{\hat{\beta}}
p^{2}\left(\hat{\alpha}^{2}p+2\right)^{2}
\sum\limits_{j=1}^{s-1}
\frac{\partial e_{j}(\hat{\alpha},p)}{\partial p}
\frac{\partial e_{s-j}(\hat{\alpha},p)}{\partial p} dp.
\end{multline}

Note that $e_s(\hat{\alpha},0)=0$ for all $s\ge1$. Using an expansion similar to \eqref{even}, the even coefficients $e_{2s}(\hat{\alpha},\hat{\beta})$ ($s=1,2,3,\ldots$), viewed as functions of $\zeta$, are meromorphic in a neighborhood of $\zeta=\pm\hat{\alpha}$ (hence may have poles only at those points). Moreover, by \cref{eq58} and induction on $s$, the odd coefficients $e_{2s+1}(\hat{\alpha},\hat{\beta})$ ($s=0,1,2,\ldots$) have a branch cut on $[-\hat{\alpha},\hat{\alpha}]$ and are analytic on $\mathbb{C}\setminus[-\hat{\alpha},\hat{\alpha}]$.

We can now apply \cite[Thm. 1.1]{Dunster:2020:LGE} to obtain solutions $\hat{Y}=(\zeta^{2}-\alpha^{2})^{-1/4}\hat{W}$ of \cref{eq48a}, where $\hat{W}$ is one of the following four LG expansions
\begin{equation}
\label{eq59}
\hat{W}_{j}(u,\hat{\alpha},\zeta) =
\exp \left\{u\hat{\xi}
+\sum\limits_{s=1}^{n-1}{
\frac{e_{s}(\hat{\alpha},\hat{\beta})}
{u^{s}}}\right\} \left\{1+\hat{\eta}_{n,j}
(u,\hat{\alpha},\zeta) \right\}
\quad (j=1,3),
\end{equation}
\begin{equation} 
\label{eq60}
\hat{W}_{2}(u,\hat{\alpha},\zeta) =
\exp \left\{-u\hat{\xi}
+\sum\limits_{s=1}^{n-1}{(-1)^{s}
\frac{e_{s}(\hat{\alpha},\hat{\beta})}
{u^{s}}}
\right\} \left\{1+\hat{\eta}_{n,2}(u,\hat{\alpha},\zeta) \right\},
\end{equation}
and
\begin{equation} 
\label{eq61}
\hat{W}_{4}^{\pm}(u,\hat{\alpha},\zeta) =
\exp \left\{-u\hat{\xi}
+\sum\limits_{s=1}^{n-1}{(-1)^{s}
\frac{e_{s}(\hat{\alpha},\hat{\beta})}
{u^{s}}}
\right\} 
\left\{1+\hat{\eta}_{n,4}^{\pm}(u,\hat{\alpha},\zeta) \right\}.
\end{equation}
Like the solutions given by \eqref{eq27}, the LG solutions $\hat{W}_4^\pm(u,\hat\alpha,\zeta)$ are proportional but have differing $\zeta$-domains of validity, namely the regions $\Theta_4^\pm$ where $\hat\eta_{n,4}^\pm(u,\hat\alpha,\zeta)=\mathcal{O}(u^{-n})$ as $u\to\infty$.

Note all of these functions are independent of $n$ due to their (unique) behavior at infinity in sectors where they are recessive. For example, since $\hat{\beta} \to 0$ as $|\zeta| \to \infty$, $e_{s}(\hat{\alpha},0)=0$, and from the error bounds provided below $\hat{\eta}_{n,1}(u,i \infty)=0$, we have from \cref{eq59}
\begin{equation}
\label{eq62}
\hat{W}_{1}(u,\hat{\alpha},\zeta) \sim
e^{u\hat{\xi}}
\quad (\zeta \to i \infty, \, \Re(\hat{\xi}) \to -\infty),
\end{equation}
and similarly for the other LG solutions.

For $\zeta \in \Theta_{j}$ ($j=1,2,3$) we again apply \cite[Thm. 1.1]{Dunster:2020:LGE}, and consequently bounds on the error terms are given by
\begin{equation} 
\label{eq63}
\left\vert \hat{\eta}_{n,j}(u,\hat{\alpha},\zeta) \right\vert 
\leq |u|^{-n}
\omega_{n,j}(u,\hat{\alpha},\hat{\beta}) 
\exp \left\{|u|^{-1}\varpi_{n,j}(u,\hat{\alpha},\hat{\beta}) +|u|^{-n}\omega_{n,j}(u,\hat{\alpha},\hat{\beta}) \right\},
\end{equation}
where
\begin{multline} 
\label{eq64}
\omega_{n,j}(u,\hat{\alpha},\hat{\beta}) 
=2\int_{0}^{\hat{\beta}}
{\left\vert
\frac{\partial e_{n}(\hat{\alpha},p)}{\partial p} d p
\right\vert } \\ 
+\sum\limits_{s=1}^{n-1}\frac{1}{|u|^{s}}
 \sum\limits_{k=s}^{n-1}
\int_{0}^{\hat{\beta}}
{\left\vert p^{2}\left(\hat{\alpha}^{2}p+2\right)^{2}
\frac{\partial e_{k}(\hat{\alpha},p)}{\partial p}
\frac{\partial e_{s+n-k-1}(\hat{\alpha},p)}{\partial p} dp
\right\vert },
\end{multline}
and
\begin{equation} 
\label{eq65}
\varpi _{n,j}(u,\hat{\alpha},\hat{\beta}) =4\sum\limits_{s=0}^{n-2}
\frac{1}{|u|^{s}}
\int_{0}^{\hat{\beta}}
\left\vert
\frac{\partial e_{s+1}(\hat{\alpha},p)}{\partial p} dp 
\right\vert .
\end{equation}
These bounds also hold for $\hat{\eta}_{n,4}^{\pm}(u,\hat{\alpha},\zeta)$ for $\zeta \in \Theta_{4}^{\pm}$ .

Integration is taken along paths that correspond to paths in the $\hat{\xi}$ plane as described above in the definitions of the regions of validity. The lower integration limits $\hat{\beta} =0$ are chosen to correspond to $\zeta = \pm i \infty$ for $j=1,3$, respectively, $\zeta = +\infty$ for $j=2$, and $\zeta = -\infty \pm i0$ for the bounds for $\hat{\eta}_{n,4}^{\pm}(u,\hat{\alpha},\zeta)$.

Matching solutions recessive at $\zeta= \pm \infty$ and $\zeta= \pm i \infty$, and using \cref{eq15,eq50}, provides the following:
\begin{equation} 
\label{eq66}
U\left(\tfrac12 u\hat{\alpha}^2,-i\sqrt{2u}\,\zeta\right)
=\left( \frac{2e}{u \hat{\alpha}^2} \right)^{\frac14 u\hat{\alpha}^2}
\frac{e^{\frac{1}{4}(u\hat{\alpha}^2+1)\pi i}}
{\left\{2u\left(\zeta^{2}-\hat{\alpha}^{2}\right)
\right\}^{1/4}}\hat{W}_{1}(u,\hat{\alpha},\zeta),
\end{equation}
\begin{equation} 
\label{eq67}
U\left(\tfrac12 u\hat{\alpha}^2,i\sqrt{2u}\,\zeta\right)
=\left( \frac{2e}{u \hat{\alpha}^2} \right)^{\frac14 u\hat{\alpha}^2}
\frac{e^{-\frac{1}{4}(u\hat{\alpha}^2+1)\pi i}}
{\left\{2u\left(\zeta^{2}-\hat{\alpha}^{2}\right)
\right\}^{1/4}}\hat{W}_{3}(u,\hat{\alpha},\zeta),
\end{equation}
\begin{equation} 
\label{eq68}
U\left(-\tfrac12 u\hat{\alpha}^2,\sqrt{2u}\,\zeta\right)
=\left(\frac{u \hat{\alpha}^2}{2e}\right)
^{\frac14 u\hat{\alpha}^2}
\frac{1}{\left\{2u\left(\zeta^{2}-\hat{\alpha}^{2}\right)
\right\}^{1/4}}\hat{W}_{2}(u,\hat{\alpha},\zeta),
\end{equation}
and
\begin{multline} 
\label{eq69}
U\left(-\tfrac12 u\hat{\alpha}^2,-\sqrt{2u}\,\zeta\right)
=\pm i e^{\mp \frac{1}{2}u\hat{\alpha}^2\pi i}
\left(\frac{u \hat{\alpha}^2}{2e}\right)
^{\frac14 u\hat{\alpha}^2}
\\ \times
\frac{1}{\left\{2u\left(\zeta^{2}-\hat{\alpha}^{2}\right)
\right\}^{1/4}}\hat{W}_{4}^{\pm}(u,\hat{\alpha},\zeta),
\end{multline}
where we recall that the coefficients $e_{s}(\hat{\alpha},\hat{\beta})$ vanish at $|\zeta| = \infty$ ($\hat{\beta}=0$).

Next, for the derivatives, from (\ref{eq48a}) we find that
\begin{equation} 
\label{eq70}
\frac{d^{2}\tilde{Y}}{d\zeta^{2}}=
\left\{u^{2}\left(\zeta^{2}-\hat{\alpha}^{2}\right)
+\frac {2\zeta^{2}+\hat{\alpha}^{2}}
{\left(\zeta^{2}-\hat{\alpha}^{2}\right)^{2}}
\right\}\tilde{Y},
\end{equation}
has solutions given by
\begin{equation} 
\label{eq71}
\tilde{Y}=(\zeta^{2}-\hat{\alpha}^{2})^{-1/2}\partial \hat{Y} /\partial \zeta,
\end{equation}
where $\hat{Y}=U(-\frac12 u\hat{\alpha}^2,\sqrt{2u}\,\zeta)$, or any other solution of (\ref{eq48a}). Now, again with $f=\zeta^{2}-\hat{\alpha}^{2}$, but now with $g=(2\zeta^{2}+\hat{\alpha}^{2})/(\zeta^{2}-\hat{\alpha}^{2})^{2}$ instead of $g=0$, we have in place of $\hat{\Psi}(\hat{\alpha},\zeta)$ as given by (\ref{eq51}) the Schwarzian, denoted by $\tilde{\Psi}(\hat{\alpha},\zeta)$, given by
\begin{equation} 
\label{eq72}
\tilde{\Psi}(\hat{\alpha},\zeta)
=\frac{5\zeta^{2}+2\hat{\alpha}^{2}}
{4\left(\zeta^{2}-\hat{\alpha}^{2}\right)^{3}}.
\end{equation}
Then, similarly to \cref{eq56,eq57,eq58}, we obtain LG coefficients given by
\begin{equation} 
\label{eq73}
\tilde{e}_{1}(\hat{\alpha},\hat{\beta})=-\tfrac{1}{24}\hat{\beta}
\left(7\hat{\alpha}^{4}\hat{\beta}^{2}
+21\hat{\alpha}^{2}\hat{\beta}+15\right),
\end{equation}
and
\begin{equation} 
\label{eq74}
\tilde{e}_{2}(\hat{\alpha},\hat{\beta})=-\tfrac{1}{16}\hat{\beta}^{2}
\left(\hat{\alpha}^{2}\hat{\beta}+2\right)^{2}
\left(7\hat{\alpha}^{4}\hat{\beta}^{2}
+14\hat{\alpha}^{2}\hat{\beta}+5\right),
\end{equation}
with $\hat{\beta}$ also given by (\ref{eq52}), and subsequent coefficients by (\ref{eq58}). Again, $\tilde{e}_{2s}(\hat{\alpha},\hat{\beta})$ are meromorphic at $\zeta = \pm \hat{\alpha}$ when regarded as functions of $\zeta$, and  $\tilde{e}_{2s+1}(\hat{\alpha},\hat{\beta})$ are analytic for $\zeta \in \mathbb{C} \setminus [-\hat{\alpha},\hat{\alpha}]$.

The LG solutions to (\ref{eq70}) are given by $\tilde{Y}_{j}(u,\hat{\alpha},\zeta)=(\zeta^{2}-\hat{\alpha}^{2})^{-1/4}\tilde{W}_{j}(u,\hat{\alpha},\zeta)$ where
\begin{equation}
\label{eq75}
\tilde{W}_{j}(u,\hat{\alpha},\zeta) =
\exp \left\{u\hat{\xi}
+\sum\limits_{s=1}^{n-1}{\frac{\tilde{e}_{s}(\hat{\alpha},\hat{\beta})}
{u^{s}}}\right\} \left\{1+\tilde{\eta}_{n,j}
(u,\hat{\alpha},\zeta) \right\}
\quad (j=1,3),
\end{equation}
\begin{equation} 
\label{eq76}
\tilde{W}_{2}(u,\hat{\alpha},\zeta) =
\exp \left\{-u\hat{\xi}
+\sum\limits_{s=1}^{n-1}{(-1)^{s}
\frac{\tilde{e}_{s}(\hat{\alpha},\hat{\beta})}{u^{s}}}
\right\} \left\{1+\tilde{\eta}_{n,2}
(u,\hat{\alpha},\zeta) \right\},
\end{equation}
and
\begin{equation} 
\label{eq77}
\tilde{W}_{4}^{\pm}(u,\hat{\alpha},\zeta) =
\exp \left\{-u\hat{\xi}
+\sum\limits_{s=1}^{n-1}{(-1)^{s}
\frac{\tilde{e}_{s}(\hat{\alpha},\hat{\beta})}{u^{s}}}
\right\} \left\{1+\tilde{\eta}_{n,4}^{\pm}
(u,\hat{\alpha},\zeta) \right\},
\end{equation}
again all being independent of $n$. Error bounds and regions of validity are the same as those for the corresponding non-derivative functions, except of course with the different coefficients.

Now, using \cref{eq16,eq50}, we can match solutions that are recessive at $\zeta= \pm \infty$ and $\zeta= \pm i \infty$, and as a result we arrive at our desired expansions:
\begin{equation} 
\label{eq78}
U'\left(\tfrac12 u\hat{\alpha}^2,-i\sqrt{2u}\,\zeta\right)
=-e^{\frac{1}{4}(u\hat{\alpha}^2-1)\pi i}
\left(\frac{2e}{u \hat{\alpha}^2}\right)^{\frac14 u\hat{\alpha}^2}
\left\{\frac18 u\left(\zeta^{2}-\hat{\alpha}^{2}\right)
\right\}^{1/4}\tilde{W}_{1}(u,\hat{\alpha},\zeta),
\end{equation}
\begin{equation} 
\label{eq79}
U'\left(\tfrac12 u\hat{\alpha}^2,i\sqrt{2u}\,\zeta\right)
=-e^{-\frac{1}{4}(u\hat{\alpha}^2-1)\pi i}
\left(\frac{2e}{u \hat{\alpha}^2}\right)^{\frac14 u\hat{\alpha}^2}
\left\{\frac18 u\left(\zeta^{2}-\hat{\alpha}^{2}\right)
\right\}^{1/4}\tilde{W}_{3}(u,\hat{\alpha},\zeta),
\end{equation}
\begin{equation} 
\label{eq80}
U'\left(-\tfrac12 u\hat{\alpha}^2,\sqrt{2u}\,\zeta\right)
=-\left(\frac{u \hat{\alpha}^2}{2e}\right)
^{\frac14 u\hat{\alpha}^2}
\left\{\frac18 u\left(\zeta^{2}-\hat{\alpha}^{2}\right)
\right\}^{1/4}\tilde{W}_{2}(u,\hat{\alpha},\zeta),
\end{equation}
and
\begin{multline} 
\label{eq81}
U'\left(-\tfrac12 u\hat{\alpha}^2,-\sqrt{2u}\,\zeta\right)
=\pm i e^{\mp \frac{1}{2}u\hat{\alpha}^2\pi i}
\left(\frac{u \hat{\alpha}^2}{2e}\right)
^{\frac14 u\hat{\alpha}^2}
\\ \times
\left\{\frac18 u\left(\zeta^{2}-\hat{\alpha}^{2}\right)
\right\}^{1/4}\tilde{W}_{4}^{\pm}(u,\hat{\alpha},\zeta).
\end{multline}

\section{Expansions valid at the turning points}
\label{sec:PCFExpansions}

Here we use the results of the previous two sections to construct asymptotic expansions for solutions of (\ref{eq01}) that are valid at the two turning points. To this end, let $\zeta$ now depend on $a$ and $z$, as defined by \cref{eq02,eq08}. Then we first define two slowly varying functions, which we denote by $\mathcal{A}(u,a,z)$ and $\mathcal{B}(u,a,z)$, by the pair of equations
\begin{multline}
\label{eq82}
w_{1}(u,a,z)
=e^{-\frac{1}{4}(u\hat{\alpha}^2-1)\pi i}\Biggl\{
U\left(\tfrac12 u\hat{\alpha}^2,-i\sqrt{2u}\,\zeta\right) 
\mathcal{A}(u,a,z)  \Biggr.
\\ \left.
+\frac{\partial U\left(\tfrac12 u\hat{\alpha}^2,-i\sqrt{2u}
\,\zeta\right)}{\partial\zeta}\mathcal{B}(u,a,z)\right\},
\end{multline}
and
\begin{multline}
\label{eq83}
w_{3}(u,a,z)
=e^{\frac{1}{4}(u\hat{\alpha}^2-1)\pi i}\Biggl\{
U\left(\tfrac12 u\hat{\alpha}^2,i\sqrt{2u}\,\zeta\right) 
\mathcal{A}(u,a,z) \Biggr.
\\ \left.
+\frac{\partial U\left(\tfrac12 u\hat{\alpha}^2,i\sqrt{2u}
\,\zeta\right)}{\partial\zeta}\mathcal{B}(u,a,z)\right\},
\end{multline}
where $w_{1}(u,a,z)$ and $w_{3}(u,a,z)$ are the solutions defined by \cref{eq18,eq19}, which are recessive at the singularities $z=x_{1}$ and $z=x_{3}$, respectively. Here, $\hat{\alpha}$ will be specified shortly. We note that $U(\tfrac12 u\hat{\alpha}^2,-i\sqrt{2u}\,\zeta)$ is recessive at $\zeta = i \infty$, which corresponds to $z=x_{1}$, and thus the RHS of (\ref{eq82}) matches the recessive behavior of the function on the LHS of the equation. Likewise, both sides of (\ref{eq83}) have a matching recessive behavior at $z=x_{3}$ ($\zeta = -i \infty$).

The multiplicative constants on the RHS of \cref{eq82,eq83} were chosen to produce, on referring to \cref{eq45,eq29},
\begin{multline}
\label{eq84}
w_{2}(u,a,z)
=\frac{\sqrt{2\pi}}
{\Gamma\left(\tfrac12 u\hat{\alpha}^2+\tfrac{1}{2}\right)}
\Biggl\{
U\left(-\tfrac12 u\hat{\alpha}^2,\sqrt{2u}\,\zeta\right) 
\mathcal{A}(u,a,z) \Biggr.
\\ \left.
+\frac{\partial U\left(-\tfrac12 u\hat{\alpha}^2,\sqrt{2u}
\,\zeta\right)}{\partial\zeta}\mathcal{B}(u,a,z)\right\}.
\end{multline}
Again we see that both sides have the same recessive behavior at a singularity, in this case at $z=x_{2}$ ($\zeta = +\infty$).

We need a similar representation for $w_{4}(u,a,z)$ that involves $U(-\tfrac12 u\hat{\alpha}^2,-\sqrt{2u}\,\zeta)$ and its derivative, both of which are recessive at $z=x_{4}$ ($\zeta=-\infty$). This is achieved by defining $\hat{\alpha}$ by
\begin{equation}
\label{eq85}
\tfrac12 u\hat{\alpha}^2=\lambda+\tfrac12,
\end{equation}
where $\lambda$ is the connection coefficient that is defined via (\ref{eq30}). Then from \cref{eq30,eq45,eq82,eq83,eq85} we arrive at our desired representation
\begin{multline}
\label{eq86}
w_{4}(u,a,z)
=\frac{\sqrt{2\pi}}
{\Gamma\left(\tfrac12 u\hat{\alpha}^2+\tfrac{1}{2}\right)}
\Biggl\{
U\left(-\tfrac12 u\hat{\alpha}^2,-\sqrt{2u}\,\zeta\right) 
\mathcal{A}(u,a,z) \Biggr.
\\ \left.
+\frac{\partial U\left(-\tfrac12 u\hat{\alpha}^2,
-\sqrt{2u}\,\zeta\right)}
{\partial\zeta}\mathcal{B}(u,a,z)\right\}.
\end{multline}

Solving \cref{eq82,eq84}, and using the Wronskian (\ref{eq43}), yields the explicit formulas
\begin{multline}
\label{eq87}
\mathcal{A}(u,a,z)=\frac{e^{-\frac{1}{4}(u\hat{\alpha}^{2}+1)\pi i}
\Gamma\left(\tfrac{1}{2}u\hat{\alpha}^{2}+\tfrac{1}{2}\right)}
{2\sqrt{u \pi}}
w_{2}(u,a,z)\frac{\partial 
U\left(\tfrac12 u\hat{\alpha}^2,-i\sqrt{2u}
\,\zeta\right)}{\partial\zeta}
\\
+\frac{i}{\sqrt{2u}}
w_{1}(u,a,z)
\frac{\partial
U\left(-\tfrac12 u\hat{\alpha}^2,\sqrt{2u}\,\zeta\right)}{\partial\zeta},
\end{multline}
and
\begin{multline}
\label{eq88}
\mathcal{B}(u,a,z)=-\frac{e^{-\frac{1}{4}
(u\hat{\alpha}^{2}+1)\pi i}
\Gamma\left(\tfrac{1}{2}u\hat{\alpha}^{2}+\tfrac{1}{2}\right)}
{2\sqrt{u \pi}}
w_{2}(u,a,z) 
U\left(\tfrac12 u\hat{\alpha}^2,-i\sqrt{2u}\,\zeta\right)
\\
-\frac{i}{\sqrt{2u}}
w_{1}(u,a,z)
U\left(-\tfrac12 u\hat{\alpha}^2,\sqrt{2u}\,\zeta\right).
\end{multline}
Similar representations can be obtained by choosing any pair of \cref{eq82,eq83,eq84,eq86}.

The plan is to replace the six functions $w_{j}(u,a,z)$ ($j=1,2$), $U(-\tfrac12 u\hat{\alpha}^2,\sqrt{2u}\,\zeta)$, $\partial U(-\tfrac12 u\hat{\alpha}^2,\sqrt{2u}\,\zeta)/ \partial \zeta$, $U(\tfrac12 u\hat{\alpha}^2,-i \sqrt{2u}\,\zeta)$ and $\partial U(\tfrac12 u\hat{\alpha}^2,-i \sqrt{2u}\,\zeta)/\partial \zeta$ in \cref{eq87,eq88} by their LG expansions, and likewise in other sectors containing numerically satisfactory sets of functions. This provides a method of asymptotically evaluating the two coefficient functions in certain regions described next, since each product pair in the expression involves one recessive and one dominant solution in the region under consideration.

Let us then examine the regions of asymptotic validity in this method. To this end, define
\begin{equation} 
\label{eq89}
D=\left\{Z_{1}\cap Z_{2}\right\}\cup\left\{Z_{1}\cap Z_{4}^{+}\right\}
\cup \left\{Z_{3}\cap Z_{2}\right\}\cup\left\{Z_{3}\cap Z_{4}^{-}\right\}.
\end{equation}
We assume here that $\{x_{1},x_{2}\} \in Z_{1}\cap Z_{2}$, which means that every point in this intersection lies on a progressive path linking $x_{1}$ to $x_{2}$. Similarly for the other three intersection sets.

Now $D$ surrounds, but does not contain, the interval $[z_{1},z_{2}]$. The significance of $D$ is that any numerically satisfactory pair of LG expansions is asymptotically valid in this domain. We also need to utilize LG expansions for parabolic cylinder functions, and as such we also need to exclude the points where they are not all valid, namely those corresponding to the $\zeta$ interval $[-\hat{\alpha},\hat{\alpha}]$. Now, we know that $\zeta(z_{1})=-\alpha$ and $\zeta(z_{2})=\alpha$. This leads to the following. 

\begin{definition}
\label{defn:I}
Let $z=\hat{z}_{1}$ and $z=\hat{z}_{2}$ correspond to $\zeta= \pm \hat{\alpha}$, so that $\zeta(\hat{z}_{1})=-\hat{\alpha}$ and $\zeta(\hat{z}_{2})=\hat{\alpha}$. Let $J$ be the curve in the $z$ plane with end points $\hat{z}_{1}$ and $\hat{z}_{2}$, corresponding to the $\zeta$ interval $[-\hat{\alpha},\hat{\alpha}]$. Then for any $\delta>0$ we define $I(\delta)=\{z: |z-v|<\delta,\forall v \in [z_{1},z_{2}] \cup J\}$; in other words, $I(\delta)$ consists of all points within a distance $\delta$ of $[z_{1},z_{2}] \cup J$.
\end{definition}

Note that $J$ would be the interval $[\hat{z}_{1},\hat{z}_{2}]$ if $\hat{\alpha}$ were real. We also observe that, for large $u$, the point sets $[z_{1},z_{2}]$ and $J$ are close to each other, since from \cref{eq36,eq85} $\hat{\alpha}=\alpha+\mathcal{O}(u^{-2})$ as $u \to \infty$.

For $z \in D \setminus I(\delta)$ we can then use both the LG expansions of the general equation given in \cref{sec:LG}, as well as those for parabolic cylinder functions given in \cref{sec:PCF}. Thus, we can use the appropriate LG expansions from \cref{eq02,eq18,eq19,eq20,eq27,eq28,eq49,eq59,eq60,eq61,eq66,eq67,eq68,eq69,eq75,eq76,eq77,eq78,eq79,eq80,eq81} in \cref{eq87,eq88}, and as a result we find that the same expansions are obtained regardless of which pair is chosen, namely
\begin{multline}
\label{eq90}
\mathcal{A}(u,a,z) = \frac{(2u)^{1/4}}{2}\left(\frac{\zeta^{2}-\hat{\alpha}^{2}}
{f(a,z)}\right)^{1/4}
\\ \times
\left[K_{n,2}(u,a)\exp\left\{u(\hat{\xi}-\xi)
+\sum_{s=1}^{n-1}
(-1)^{s}\frac{\tilde{\mathcal{E}}_{s}(a,z)}{u^{s}}\right\}
\right.
\\
\left.
-iK_{n,1}(u,a)
\exp\left\{-u(\hat{\xi}-\xi)
+\sum_{s=1}^{n-1}\frac{\tilde{\mathcal{E}}_{s}(a,z)}{u^{s}}\right\}
\right]\left\{1+ \mathcal{O}
\left(\frac{1}{u^{n}}\right)\right\},
\end{multline}
and
\begin{multline}
\label{eq91}
\mathcal{B}(u,a,z) = 
-\frac{1}{(2u)^{3/4}}\frac{1}
{\left\{f(a,z)\left(\zeta^{2}-\hat{\alpha}^{2}\right)
\right\}^{1/4}}
\\ \times
\left[K_{n,2}(u,a)
\exp\left\{u(\hat{\xi}-\xi)
+\sum_{s=1}^{n-1}(-1)^{s}
\frac{\mathcal{E}_{s}(a,z)}{u^{s}}\right\}
\right.
\\
\left.
+iK_{n,1}(u,a)
\exp\left\{-u(\hat{\xi}-\xi)
+\sum_{s=1}^{n-1}\frac{\mathcal{E}_{s}(a,z)}{u^{s}}\right\}
\right]\left\{1+ \mathcal{O}
\left(\frac{1}{u^{n}}\right)\right\},
\end{multline}
as $u \to \infty$, uniformly for $z \in D\setminus I(\delta)$, where
\begin{equation} 
\label{eq92}
K_{n,1}(u,a)
=\left(\frac{u\hat{\alpha}^{2}}{2e}\right)
^{\frac14u\hat{\alpha}^{2}}
k_{n,1}(u,a),
\end{equation}
\begin{equation} 
\label{eq93}
K_{n,2}(u,a)
=\frac{1}{\sqrt{2\pi}}
\left(\frac{2e}{u\hat{\alpha}^{2}}\right)
^{\frac14u\hat{\alpha}^{2}}
\Gamma\left(\tfrac{1}{2}u\hat{\alpha}^{2}
+\tfrac{1}{2}\right) k_{n,2}(u,a),
\end{equation}
\begin{equation} 
\label{eq94}
\tilde{\mathcal{E}}_{s}(a,z)
=\hat{E}_{s}(a,z)
+(-1)^{s}\tilde{e}_{s}(\hat{\alpha},\hat{\beta}),
\end{equation}
and
\begin{equation} 
\label{eq95}
\mathcal{E}_{s}(a,z)
=\hat{E}_{s}(a,z)
+(-1)^{s}e_{s}(\hat{\alpha},\hat{\beta}).
\end{equation}

Now, as we observed above, $\hat{\alpha}^{2}=\alpha^{2}+\mathcal{O}(u^{-2})$ as $u \to \infty$. Thus, from \cref{eq49,eq08}
\begin{multline} 
\label{eq96}
\hat{\xi}-\xi
=\frac{1}{2}\zeta\left\{
\left(\zeta^{2}-\hat{\alpha}^{2}\right)^{1/2}
-\left(\zeta^{2}-\alpha^{2}\right)^{1/2}\right\}
-\frac{1}{2}\hat{\alpha}^{2}
\ln\left\{\zeta+\left(\zeta^{2}
-\hat{\alpha}^{2}\right)^{1/2}\right\}
\\
+\frac{1}{2}\alpha^{2}\ln\left\{\zeta+\left(\zeta^{2}-\alpha^{2}\right)^{1/2}\right\}
+\frac{1}{2}\hat{\alpha}^{2}\ln(\hat{\alpha})
-\frac{1}{2}\alpha^{2}\ln(\alpha)
=\mathcal{O}\left(\frac{1}{u^{2}}\right),
\end{multline}
for $\alpha \in [0,\alpha_{1}]$ and bounded $\zeta$. It can also be confirmed that $\hat{\xi}-\xi=u^{-2}\mathcal{O}\{\ln(\zeta)\}+\mathcal{O}(u^{-1})$ as $\zeta \to \infty$. Therefore, the terms $\pm u(\hat{\xi}-\xi)$ appearing in the exponents in \cref{eq90,eq91} are bounded for fixed $\zeta$ as $u \to \infty$, and contribute at worst (on exponentiation) to a slowly growing algebraic term as $\zeta \to \infty$ with $u$ large. This confirms that $\mathcal{A}(u,a,z)$ and $\mathcal{B}(u,a,z)$ are slowly-varying as $u \to \infty$ and $z \in D \setminus I(\delta)$.

The above expansions break down near the turning points, and indeed are not valid in $I(\delta)$. Following \cite{Dunster:2017:COA} we can instead use Cauchy's integral formula to obtain expansions that are valid at these points. Thus, because the functions are analytic at these turning points, we have
\begin{equation} 
\label{eq97}
\mathcal{A}(u,a,z)
=\frac{1}{2 \pi i}\oint_{\mathcal{C}}
\frac{\mathcal{A}(u,a,t)}{t-z}dt,
\,
\mathcal{B}(u,a,z)
=\frac{1}{2 \pi i}\oint_{\mathcal{C}}
\frac{\mathcal{B}(u,a,t)}{t-z}dt.
\end{equation}
The contour $\mathcal{C}$ is a suitably chosen positively orientated simple loop lying in $D$ that encloses $z$ and the set $I(\delta)$. In the integrands we can then insert the expansions \cref{eq90,eq91}, since they are valid on the contour. This enables an accurate evaluation of $\mathcal{A}(u,a,z)$ and $\mathcal{B}(u,a,z)$ for points inside the contour, particularly those in $I(\delta)$.

Note that \cref{eq82,eq83,eq84,eq86}, being exact, can be differentiated with respect to $z$ to produce expansions for the derivatives of the solutions. In doing so, one uses the chain rule in conjunction with \cref{eq05}, as well as \cref{eq48a}. For example, from \cref{eq82}
\begin{multline}
\label{eq98}
\frac{\partial w_{1}(u,a,z)}{\partial z}
=e^{-\frac{1}{4}(u\hat{\alpha}^2-1)\pi i}\Biggl\{
U\left(\tfrac12 u\hat{\alpha}^2,-i\sqrt{2u}\,\zeta\right) 
\tilde{\mathcal{A}}(u,a,z)  \Biggr.
\\ \left.
+\frac{\partial U\left(\tfrac12 u\hat{\alpha}^2,-i\sqrt{2u}
\,\zeta\right)}{\partial\zeta}
\tilde{\mathcal{B}}(u,a,z)\right\},
\end{multline}
where
\begin{equation} 
\label{eq99}
\tilde{\mathcal{A}}(u,a,z)=\frac{\partial 
\mathcal{A}(u,a,z)}{\partial z}
+u^{2}\left\{\left(\zeta^{2}-\alpha^{2}\right)f(u,a,z)\right\}^{1/2}
\mathcal{B}(u,a,z),
\end{equation}
and
\begin{equation} 
\label{eq100}
\tilde{\mathcal{B}}(u,a,z)=\frac{\partial \mathcal{B}(u,a,z)}{\partial z}
+\left\{\frac{f(u,a,z)}{\zeta^{2}-\alpha^{2}}\right\}^{1/2}
\mathcal{A}(u,a,z).
\end{equation}
Similarly for the other solutions. Then, for $z$ on or close to the set $I(\delta)$ we can use, in conjunction with (\ref{eq97}),
\begin{equation} 
\label{eq101}
\frac{\partial \mathcal{A}(u,a,z)}{\partial z}
=\frac{1}{2 \pi i}\oint_{\mathcal{C}}
\frac{\mathcal{A}(u,a,t)}{(t-z)^{2}}dt,
\,
\frac{\partial \mathcal{B}(u,a,z)}{\partial z}
=\frac{1}{2 \pi i}\oint_{\mathcal{C}}
\frac{\mathcal{B}(u,a,t)}{(t-z)^{2}}dt.
\end{equation}

We shall use the following theorem, the proof of which is a straightforward modification of the one given for a closely related theorem \cite[Thm. 3.1]{Dunster:2021:NKF} that holds for isolated singularities.

\begin{theorem} 
\label{thm:nopoles}
Let $0<\rho_{0}<\rho_{1}$, $u>0$, $z_{0}\in \mathbb{C}$, $H(u,z)$ be an analytic function of $z$ in the open disk $D_{1}=\{z:\, |z-z_{0}|<\rho_{1}\}$, and $h_{s}(z)$ ($s=0,1,2,\ldots$) be a sequence of functions that are analytic in the annulus $R=\{z:\, \rho_{0}<|z-z_{0}|<\rho_{1}\}$. If $H(u,z)$ is known to possess the asymptotic expansion
\begin{equation} 
\label{eq102}
H(u,z) \sim \sum\limits_{s=0}^{\infty}\frac{h_{s}(z)}{u^{s}}
\quad (u \rightarrow \infty),
\end{equation}
in $R$, then each $h_{s}(z)$ is analytic in $D_{1}$ (except possibly having removable singularities there), and the expansion (\ref{eq102}) actually holds for all $z \in D_{1}$.
\end{theorem}

The following provides traditional asymptotic expansions under certain conditions, and we shall illustrate this in the next section.

\begin{theorem}
\label{thm:tp-series}
For arbitrary positive integer $n$, and $a_{0} \leq a \leq a_{1}$, suppose that
\begin{equation}
\label{eq103}
\frac{K_{n,1}(u,a)}{K_{n,2}(u,a)}
= l_{0}(a)\exp\left\{ \sum_{s=1}^{n-1}
\frac{l_{s}(a)}{u^{s}} \right\}
\left\{1+ \mathcal{O}
\left(\frac{1}{u^{n}}\right)\right\}
\quad (u \to \infty),
\end{equation}
for certain constants $l_{2s+1}(a)$ with $l_{0}(a) \neq 0$, where $K_{n,1}(u,a)$ and $K_{n,2}(u,a)$ are defined by \cref{eq92,eq93}. Also, assume for some fixed $z_{0} \in \mathbb{C}$, positive $\rho_{0}$, and sufficiently small $\delta>0$, that the disk $D_{0}=\{z:\, |z-z_{0}|<\rho_{0}\}$ contains $I(\delta)$ and lies inside $D \cup I(\delta)$. Then as $u \to \infty$, $\mathcal{A}(u,a,z)$ and $\mathcal{B}(u,a,z)$ possess asymptotic expansions of the forms
\begin{equation}
\label{eq104}
\mathcal{A}(u,a,z) \sim 
\Lambda(u,a)\left(\frac{\zeta^{2}-\hat{\alpha}^{2}}
{f(a,z)}\right)^{1/4}
\left\{
1 + \sum_{s=1}^{\infty}
\frac{\mathrm{A}_{s}(a,z)}{u^{s}}
\right\},
\end{equation}
and
\begin{equation}
\label{eq105}
\mathcal{B}(u,a,z) \sim 
\frac{\Lambda(u,a)}{u^{2}}
\left(\frac{\zeta^{2}-\hat{\alpha}^{2}}
{f(a,z)}\right)^{1/4}
\sum_{s=0}^{\infty}
\frac{\mathrm{B}_{s}(a,z)}{u^{s}},
\end{equation}
uniformly for  $z \in D \cup I(\delta)$ and $a_{0} \leq a \leq a_{1}$; here $\Lambda(u,a)$ is a constant that can be explicitly determined from the connection coefficients. In addition, $\mathrm{A}_{s+1}(a,z)$ and $\mathrm{B}_{s}(a,z)$ ($s=0,1,2,\ldots$) are analytic in $D \cup I(\delta)$.
\end{theorem}

\begin{remark}
From \cref{eq10}, \cref{eq36}, \cref{eq85}, \cref{eq91}, \cref{eq92}, \cref{lem:ratio}, and \cite[Eq. 5.11.1]{NIST:DLMF}, it follows that
\begin{equation}
\label{eq106}
\frac{K_{n,1}(u,a)}{K_{n,2}(u,a)}
= i \exp\left\{ \sum_{s=0}^{n-1}
\frac{l_{2s+1}(a)}{u^{2s+1}} \right\}
\left\{1+ \mathcal{O}
\left(\frac{1}{u^{2n+1}}\right)\right\}
\quad (u \to \infty),
\end{equation}
under the weaker condition $a_{0}<a_{0}+\delta \leq a \leq a_{1}$. 
\end{remark}

\begin{proof}

Firstly, from \cref{eq18,eq20,eq24,eq25,eq26},
\begin{equation} 
\label{eq107}
f^{1/4}(a,z) e^{u \xi}w_{2}(u,a,z) 
\to \gamma_{2}(u,a)
\quad (z \to x_{2}),
\end{equation}
where $\gamma_{2}(u,a)$ is given by
\begin{equation}
\label{eq108}
\gamma_{2}(u,a) 
=  k_{n,2}(u,a) \exp \left\{\sum\limits_{s=1}^{n-1}{(-1)^{s}
\frac{\hat{E}_{s}(a,x_{2})}{u^{s}}}
\right\} \left\{1+\eta_{n,2}(u,a,x_{2}) \right\},
\end{equation}
for arbitrary positive integer $n$; observe from \cref{eq107} that $\gamma_{2}(u,a)$ is independent of $n$. Hence from \cref{eq93}
\begin{multline} 
\label{eq109}
K_{n,2}(u,a)
= \frac{1}{\sqrt{2\pi}}
\left(\frac{2e}{u\hat{\alpha}^{2}}\right)
^{\frac14u\hat{\alpha}^{2}}
\Gamma\left(\tfrac{1}{2}u\hat{\alpha}^{2}
+\tfrac{1}{2}\right) \gamma_{2}(u,a)
\\ \times
\exp \left\{\sum\limits_{s=1}^{n-1}{(-1)^{s-1}
\frac{\hat{E}_{s}(a,x_{2})}{u^{s}}}
\right\} 
\left\{1+ \mathcal{O}
\left(\frac{1}{u^{n}}\right)\right\}
\quad (u \to \infty).
\end{multline}
Now rewrite \cref{eq90} in the form
\begin{multline}
\label{eq110}
\mathcal{A}(u,a,z) = 
\frac{(2u)^{1/4}}{2}K_{n,2}(u,a)
\left(\frac{\zeta^{2}-\hat{\alpha}^{2}}
{f(a,z)}\right)^{1/4}
\\ \times
\left[\exp\left\{u(\hat{\xi}-\xi)
+\sum_{s=1}^{n-1}
(-1)^{s}\frac{\tilde{\mathcal{E}}_{s}(a,z)}{u^{s}}\right\}
\right.
\\
\left.
-i \frac{K_{n,1}(u,a)}{K_{n,2}(u,a)}
\exp\left\{-u(\hat{\xi}-\xi)
+\sum_{s=1}^{n-1}\frac{\tilde{\mathcal{E}}_{s}(a,z)}{u^{s}}\right\}
\right]\left\{1+ \mathcal{O}
\left(\frac{1}{u^{n}}\right)\right\}.
\end{multline}
Then, for $z \in D \setminus I(\delta)$ and $a_{0} \leq a \leq a_{1}$, one can show, using \cref{eq02,eq49,eq36,eq40,eq103,eq109}, that \cref{eq90} can be expanded into a formal asymptotic series of the form \cref{eq104}. The expansion \cref{eq105} can be obtained in a similar way.

Next, it follows from \cref{eq02,eq08,eq49,eq94,eq95} (and see also the comments after \cref{eq58,eq74}) that $\hat{\xi}$, $\xi$, $\tilde{\mathcal{E}}_{s}(a,z)$ and $\mathcal{E}_{s}(a,z)$ are analytic in $D$ except the first two (regarded as functions of $z$), and possibly the odd coefficients of the latter two, having points of discontinuity on the real axis to the left of $z=z_{1}$. As such, all are certainly analytic for $z \in D \setminus \{[x_{4},z_{1}] \cup I(\delta)\}$.

Consider now an annulus $R=\{z:\, \rho_{0}<|z-z_{0}|<\rho_{1}\} \subset D$, with $\rho_{0}$ defined by the hypothesis of the theorem. The definition of $\rho_{0}$ implies that $I(\delta)$ lies inside $D_{0}$ for sufficiently small positive $\delta$, and therefore does not lie within $R$. Thus, from above, $\hat{\xi}$, $\xi$, $\tilde{\mathcal{E}}_{s}(a,z)$ and $\mathcal{E}_{s}(a,z)$ are analytic in $R \setminus [-\rho_{1},-\rho_{0}]$, and are bounded in $R$.

Furthermore, from the re-expansion of \cref{eq110} into \cref{eq104}, and similarly from the construction of \cref{eq105}, which are both valid in $R$, we infer that each $\mathrm{A}_{s}(a,z)$ and $\mathrm{B}_{s}(a,z)$ is a polynomial in terms of $\hat{\xi}$, $\xi$, $\tilde{\mathcal{E}}_{s}(a,z)$ and $\mathcal{E}_{s}(a,z)$. Thus they too must be analytic in the cut annulus $R\setminus [-\rho_{1},-\rho_{0}]$, and bounded in $R$. However, $\mathcal{A}(u,a,z)$ and $\mathcal{B}(u,a,z)$ are analytic in $R$, and in particular continuous at $[-\rho_{1},-\rho_{0}]$, and so from \cref{eq104,eq105} we conclude that $\mathrm{A}_{s}(a,z)$ and $\mathrm{B}_{s}(a,z)$ must also be continuous at this interval. Accordingly, these coefficients are analytic at the cut, and hence throughout $R$.

To further extend this to $D_{1}=\{z:\, |z-z_{0}|<\rho_{1}\}$, and in particular to $I(\delta)$, we appeal to \cref{thm:nopoles}. From this it follows under the above assumptions that the coefficients $\mathrm{A}_{s}(a,z)$ and $\mathrm{B}_{s}(a,z)$ are actually analytic in $D_{1}$, which contains both turning points (whether separated or not). Moreover, the asymptotic expansions \cref{eq104,eq105} hold in this disk, and by extension to all points in $D \cup I(\delta)$. 
\end{proof}

Let us summarize the main results of this section. Consider the differential equation \cref{eq01} with the assumptions on $f(a,z)$ and $g(z)$ stated in the paragraph following that equation, and in particular, for $a_{0} \leq a \leq a_{1}$ $f(a,z)$ has simple zeros at $z=z_{1}$ and $z=z_{2}$, which coalesce as $a \to a_{0}$. Define $\xi$, $\zeta$, and $\alpha$ by \cref{eq02,eq08,eq11}. Let coefficients $\hat{E}_{s}(a,z)$ be given by \cref{eq21,eq22,eq23}. For each positive integer $n$ let $w_{j}(u,a,z)$ ($j=1,2,3,4$) be solutions of \cref{eq01} given by \cref{eq18,eq19,eq20,eq27,eq28}, which are recessive at the singularities $z=x_{j}$ ($j=1,2,3,4$) of \cref{eq01} which correspond to $\zeta=i\infty$, $\zeta=+\infty$, $\zeta=-i\infty$, and $\zeta=-\infty$, respectively. Next, $k_{n,j}(u,a)$ ($j=1,2,3$) and $k_{n,4}^{\pm}(u,a)$ are constants, continuous for $a_{0} \leq a \leq a_{1}$, chosen such that the connection formulas \cref{eq29,eq30} hold for some $\lambda$. Define $\hat{\xi}$ and $\hat{\beta}$ by \cref{eq49,eq52,eq53}, where $\hat{\alpha}$ is given by \cref{eq85}, and let coefficients $e_{s}(\hat{\alpha},\hat{\beta})$ and $\tilde{e}_{s}(\hat{\alpha},\hat{\beta})$ be given by \cref{eq56,eq57,eq73,eq74}, with subsequent ones given recursively by \cref{eq58}. Let $D$ be a domain that surrounds, but does not contain, the set $I(\delta)$ that contains the turning points (defined by \cref{defn:I}), and is given by \cref{eq89} along with certain LG regions of validity described at the beginning of \cref{sec:LG}.

Then $w_{j}(u,a,z)$ can be expressed in the forms \cref{eq82,eq83,eq84,eq86}, where, for $a_{0} \leq a \leq a_{1}$ and $z \in D$, $\mathcal{A}(u,a,z)$ and $\mathcal{B}(u,a,z)$ possess the asymptotic expansions \cref{eq90,eq91}. In these $K_{n,1}(u,a)$, $K_{n,2}(u,a)$, $\tilde{\mathcal{E}}_{s}(a,z)$ and $\mathcal{E}_{s}(a,z)$ are given by \cref{eq92,eq93,eq94,eq95}.

For points in $I(\delta)$, which contains the two turning points, the coefficient functions $\mathcal{A}(u,a,z)$ and $\mathcal{B}(u,a,z)$ can be evaluated via the Cauchy integrals \cref{eq97}, with the expansions \cref{eq90,eq91} used on the contour, which is a simple positively orientated loop in $D$ which surrounds $z$ and both turning points.

The expansions \cref{eq90,eq91} can be expanded in the forms \cref{eq104,eq105}, under conditions given in \cref{thm:tp-series}. If applicable, this gives an alternative way to asymptotically evaluate $\mathcal{A}(u,a,z)$ and $\mathcal{B}(u,a,z)$ in $I(\delta)$.

\section{Associated Legendre functions}
\label{sec:Legendre}

The associated Legendre function of the first and second kinds $P^{-\mu}_{\nu}(z)$ and $\boldsymbol{Q}^{\mu}_{\nu}(z)$ have the explicit representations given in \cite[Eqs. 14.3.6, 14.3.7, 14.3.10]{NIST:DLMF}. These functions are solutions of the associated Legendre differential equation
\begin{equation}
\label{eq111}
(z^2-1)\frac{d^2 y}{d z^2} +2z\frac{dy}{dz} 
-\left\{\nu(\nu + 1)+\frac{\mu^2}{z^2-1}\right\}y=0,
\end{equation}
which has regular singularities at $z=\pm 1$ and $z=\infty$.

The behavior of these functions at the singularities is as follows, and throughout we assume $\mu,\nu \geq 0$. We have (see, for example, \cite[Eqs. 14.8.7 and 14.8.15]{NIST:DLMF})
\begin{equation}
\label{eq112}
P^{-\mu}_{\nu}(z)=
\frac{1}{\Gamma(\mu+1)}
\left(\frac{z-1}{2}\right)^{\mu/2}
\left\{1+\mathcal{O}(z-1)\right\}
\quad (z \rightarrow 1),
\end{equation}
\begin{equation}
\label{eq113}
P^{-\mu}_{\nu}(z)
=\frac{\Gamma\left(\nu+\tfrac{1}{2}\right)(2z)^{\nu}}
{\sqrt{\pi}\,\Gamma(\nu+\mu+1)}
\left\{1+\mathcal{O}
\left(\frac{1}{z}\right)\right\}
\quad
(z \rightarrow \infty),
\end{equation}
and
\begin{equation}
\label{eq114}
\boldsymbol{Q}^{\mu}_{\nu}(z)
=\frac{\sqrt{\pi}}
{\Gamma\left(\nu+\frac{3}{2}\right)
(2z)^{\nu+1}}
\left\{1+\mathcal{O}
\left(\frac{1}{z}\right)\right\}
\quad \left(z \rightarrow \infty \right).
\end{equation}
These two functions are recessive at the respective singularities, and they form a numerically satisfactory pair of solutions of (\ref{eq111}) in the half-plane $|\arg(z)|\leq \frac12 \pi$ (see \cite[Chap. 5, Thm. 12.1]{Olver:1997:ASF}).

Connection formulas relating the functions are required and read as follows (see \cite[Eq. 14.24.1]{NIST:DLMF}):
\begin{equation}
\label{eq115}
P^{-\mu}_{\nu}\left(ze^{\pm \pi i}\right)
=e^{\pm\nu\pi i}P^{-\mu}_{\nu}(z)
+ \frac{2
}{\Gamma(\mu-\nu)}
\boldsymbol{Q}^{\mu}_{\nu}(z),
\end{equation}
from which
\begin{equation}
\label{eq116}
2i\sin(\nu \pi)P^{-\mu}_{\nu}(z)
=P^{-\mu}_{\nu}\left(ze^{\pi i}\right)
-P^{-\mu}_{\nu}\left(ze^{-\pi i}\right).
\end{equation}
Also,
\begin{equation}
\label{eq117}
\boldsymbol{Q}^{\mu}_{\nu,\pm 1}(z)
=e^{\mp \mu\pi i}\boldsymbol{Q}^{\mu}_{\nu}(z)
\mp \frac{\pi i}{\Gamma(\nu-\mu+1)}
P^{-\mu}_{\nu}(z).
\end{equation}

Next, Ferrers functions are real-valued solutions of (\ref{eq111}), which for $z=x \in (-1,1)$ are given by (see \cite[Eqs. 14.23.4 and 14.23.5]{NIST:DLMF})
\begin{equation}
\label{eq118}
\mathsf{P}_{\nu}^{-\mu}(x)
=e^{\mp \frac12 \mu \pi i}
P_{\nu}^{-\mu}(x \pm i0),
\end{equation}
and
\begin{equation}
\label{eq119}
\mathsf{Q}_{\nu}^{-\mu}(x)
=\tfrac12 \Gamma(\nu-\mu+1)\left\{e^{\frac12\mu \pi i}
\boldsymbol{Q}_{\nu}^{\mu}(x+i0)
+e^{-\frac12\mu \pi i}\boldsymbol{Q}_{\nu}^{\mu}(x-i0)\right\},
\end{equation}
where $f(x \pm i0)=\lim_{\epsilon \to 0+}f(x \pm i\epsilon)$. From \cref{eq113,eq118} we have
\begin{equation}
\label{eq120}
\mathsf{P}^{-\mu}_{\nu}(x)=
\frac{1}{\Gamma(\mu+1)}
\left(\frac{1-x}{2}\right)^{\mu/2}
\left\{1+\mathcal{O}(x-1)\right\}
\quad (x \rightarrow 1^{-}).
\end{equation}
Thus, $\mathsf{P}_{\nu}^{-\mu}(\pm x)$ form a numerically satisfactory pair of solutions in $(-1,1)$ for $\mu \geq 0$ since $\mathsf{P}_{\nu}^{-\mu}(\pm x)$ is recessive at $x= \pm 1$. In addition, $\mathsf{P}_{\nu}^{-\mu}(x)$ and $\mathsf{Q}_{\nu}^{-\mu}(x)$ form a numerically satisfactory pair of solutions in $[0,1)$ for $\mu \geq 0$.

We require four solutions of \cref{eq111} that are recessive at $z=\pm 1$ and $z= \pm i\infty$. In terms of the general equation of \cref{sec:LG} these singularities are identified with $z= x_{2}$, $z= x_{4}$, $z= x_{1}$ and  $z= x_{3}$, respectively. We also require all four to be continuous across the cut $[-1,1]$. The first of these is $\mathsf{P}^{-\mu}_{\nu}(z)$ as defined by the hypergeometric series \cite[Eq. 14.3.1]{NIST:DLMF} for complex $x=z$ with $|z|<1$, and by analytic continuation elsewhere. This is analytic in the plane that has cuts along $(-\infty,-1]$ and $[1,\infty)$, and is, of course, equal to the real-valued Ferrers function $\mathsf{P}^{-\mu}_{\nu}(x)$ when $x \in (-1,1)$. The second fundamental solution is simply $\mathsf{P}^{-\mu}_{\nu}(-z)$, which is recessive at $z=-1$.

Next, we look for a solution that is recessive at $z=i\infty$ and continuous across the cut $[-1,1]$. The function $\boldsymbol{Q}^{\mu}_{\nu}(z)$ is indeed recessive at $z=i\infty$, but not continuous across the cut. So we define our desired function, labeled $\mathsf{Q}^{\mu}_{\nu,1}(z)$, to be equal to $\boldsymbol{Q}^{\mu}_{\nu}(z)$ in the upper half-plane and equal to the analytic continuation of that function across the cut $[-1,1]$. Consequently,
\begin{equation}
\label{eq121}
\mathsf{Q}^{\mu}_{\nu,1}(z)
=
    \begin{cases}
        \boldsymbol{Q}^{\mu}_{\nu}(z) & \Im(z) \geq 0\\
        \boldsymbol{Q}^{\mu}_{\nu,1}(z) & \Im(z) < 0
    \end{cases}\,.
\end{equation}
This fulfills the requirement of being continuous, in fact analytic, in the $z$ plane having cuts along $(-\infty,-1]$ and $[1,\infty)$, and is recessive at $z = i\infty$ (or more precisely as $z \to \infty$ for $0 \leq \arg(z) \leq \pi$).

Finally, for the fourth fundamental solution, define $\mathsf{Q}^{\mu}_{\nu,-1}(z)$ as the conjugate of $\mathsf{Q}^{\mu}_{\nu,1}(z)$, namely,
\begin{equation}
\label{eq122}
\mathsf{Q}^{\mu}_{\nu,-1}(z)
=
    \begin{cases}
        \boldsymbol{Q}^{\mu}_{\nu,-1}(z) & \Im(z) > 0\\
        \boldsymbol{Q}^{\mu}_{\nu}(z) & \Im(z) \leq 0
    \end{cases}\,,
\end{equation}
this being recessive as $z \to \infty$ for $-\pi \leq \arg(z) \leq 0$. Note that $\mathsf{Q}^{\mu}_{\nu,\pm 1}(z)$ are generally not real for $z\in (-1,1)$.

Now, from \cite[Eq. 14.23.2]{NIST:DLMF} one can show that
\begin{equation}
\label{eq123}
\mathsf{P}^{-\mu}_{\nu}(z)=\frac{i\Gamma(\nu-\mu+1)}{\pi}
\left\{e^{\frac{1}{2}\mu\pi i}\mathsf{Q}^{\mu}_{\nu,1}(z)
-e^{-\frac{1}{2}\mu\pi i}\mathsf{Q}^{\mu}_{\nu,-1}(z)\right\}.
\end{equation}
Similarly, from \cref{eq119,eq121,eq122}, we have
\begin{equation}
\label{eq124}
\mathsf{Q}_{\nu}^{-\mu}(z)
=\tfrac12 \Gamma(\nu-\mu+1)\left\{e^{\frac12\mu \pi i}
\mathsf{Q}^{\mu}_{\nu,1}(z)
+e^{-\frac12\mu \pi i}\mathsf{Q}^{\mu}_{\nu,-1}(z)\right\}.
\end{equation}
Thus from \cref{eq123,eq124}, and \cite[Eq. 14.9.10]{NIST:DLMF}
\begin{equation}
\label{eq125}
\mathsf{P}^{-\mu}_{\nu}(-z)=\frac{i\Gamma(\nu-\mu+1)}{\pi}
\left\{e^{\frac{1}{2}(2\nu-\mu)\pi i}
\mathsf{Q}^{\mu}_{\nu,1}(z)
-e^{-\frac{1}{2}(2\nu-\mu)\pi i}
\mathsf{Q}^{\mu}_{\nu,-1}(z)\right\}.
\end{equation}

From \cref{eq125,eq113,eq117,eq121,eq122,eq123} we then find that
\begin{equation}
\label{eq126}
\mathsf{P}^{-\mu}_{\nu}(z)
=\frac{e^{\mp \mu \pi i/2}
\Gamma\left(\nu+\tfrac{1}{2}\right)(2z)^{\nu}}
{\sqrt{\pi}\,\Gamma(\nu+\mu+1)}
\left\{1+\mathcal{O}
\left(\frac{1}{z}\right)\right\}
\quad
(z \rightarrow \pm i\infty),
\end{equation}
and
\begin{equation}
\label{eq127}
\mathsf{P}^{-\mu}_{\nu}(-z)
=\frac{e^{\mp (2\nu -\mu) \pi i/2}
\Gamma\left(\nu+\tfrac{1}{2}\right)(2z)^{\nu}}
{\sqrt{\pi}\,\Gamma(\nu+\mu+1)}
\left\{1+\mathcal{O}
\left(\frac{1}{z}\right)\right\}
\quad
(z \rightarrow \pm i\infty).
\end{equation}

Next, let
\begin{equation}
\label{eq128}
u=\nu+\frac12, \; 1-a^2=\frac{\mu^2}{u^2}.
\end{equation}
Then from \cref{eq111} we have $(1-z^2)^{1/2}\mathsf{P}^{-\mu}_{\nu}(\pm z)$ and $(1-z^2)^{1/2}\mathsf{Q}^{\mu}_{\nu,\pm 1}(z)$ as solutions of \cref{eq01}, where
\begin{equation}
\label{eq129}
f(a,z)=\frac{z^2-a^2}{\left(1-z^2\right)^2}, \;
g(z)=-\frac{3+z^2}{4\left (1-z^2\right )^2}.
\end{equation}

We assume for arbitrary fixed $\delta>0$ that
\begin{equation}
\label{eq130}
0 \leq a \leq 1-\delta <1,
\end{equation}
which is equivalent to $\delta' u \leq \mu \leq u$ for arbitrary fixed $\delta' \in (0,1)$. Thus, we can identify the turning points $z=z_{j}$ ($j=1,2$) of \cref{sec:Introduction} with $z=\pm a$, and from \cref{eq130} $a_{0}=0$ and $a_{1}=1-\delta<1$. Note that the turning points coalesce as $a \to 0$, which from \cref{eq128} is equivalent to $\mu=u=\nu+\frac12$. Also, the upper bound in \cref{eq130} ensures that the turning points are bounded away from the singularities $z=\pm 1$. This situation was studied in \cite{Boyd:1986:TPS}.

Now, from \cref{eq02,eq129}
\begin{multline}
\label{eq131}
\xi =\int_{a}^{z} f^{1/2}(a,t) dt
=\int_{a}^{z} \frac{\left(t^2-a^2\right)^{1/2}}
{1-t^2} dt
\\
=\frac{\mu}{2u}\ln\left\{\frac{\left(uz+\mu X+ua^2\right)
\left(uz+\mu X-ua^2\right)}{\left(1-z^2\right)}\right\}
-\ln\left\{u(z+X)\right\}
-\frac{(u-\mu)\ln(ua)}{u},
\end{multline}
where
\begin{equation}
\label{eq132}
X=\left(z^2-a^2\right)^{1/2}.
\end{equation}
The branch is such that $\xi$ is positive for $a < z <1$ and continuous in the $z$ plane having cuts along $(-\infty,a]$ and $[1,\infty)$. An equivalent representation, which is useful for $a \leq z <1$, is given by
\begin{equation}
\label{eq133}
\xi=\left(1-a^{2}\right)^{1/2}\mathrm{arctanh}
\left\{\frac{1}{z}\left(\frac{z^{2}-a^{2}}{1-a^{2}}\right)^{1/2}\right\}
-\mathrm{arccosh}\left(\frac{z}{a}\right).
\end{equation}
Note in the limit $a \to 0$
\begin{equation}
\label{eq134}
\xi =- \tfrac12
\ln \left(1-z^2\right).
\end{equation}

We next record the following limiting forms:
\begin{equation}
\label{eq135}
\xi=\frac{2\sqrt{2a}}{3\left(1-a^2\right)}
(z-a)^{3/2}+\mathcal{O}\left\{(z-a)^2\right\}
\quad (z \to a),
\end{equation}
and
\begin{multline}
\label{eq136}
\xi=-\frac12 \left(1-a^2\right)^{1/2}\ln\left\{\frac{1-z}
{2\left(1-a^2\right)}\right\}
-\ln\left\{1+\left(1-a^2\right)^{1/2}\right\}
\\
+\left\{1-\left(1-a^2\right)^{1/2}\right\}\ln(a)
+\mathcal{O}(z-1)
\quad (z \to 1),
\end{multline}
which has the following limit when $a=0$
\begin{equation}
\label{eq137}
\xi=-\tfrac{1}{2}\ln\left\{2(1-z)\right\}
+\mathcal{O}(1-z)
\quad (z \to 1,\, a=0).
\end{equation}
In addition,
\begin{multline}
\label{eq138}
\xi=-\ln(2z)
+\left\{1-\left(1-a^2\right)^{1/2}\right\}\ln(a)
+\left(1-a^2\right)^{1/2}\ln\left\{1+\left(1-a^2\right)^{1/2}\right\}
\\
\pm \frac12\left(1-a^2\right)^{1/2}  \pi i
+\mathcal{O}\left(\frac1z\right)
\quad \left(z \to \infty,\, \Im(z) \gtrless 0\right),
\end{multline}
again with the limit applying when $a=0$, that is, $\xi=-\ln(z)\pm \frac12  \pi i+\mathcal{O}(z^{-1})$.

Consider now the LG expansions \cref{eq18,eq28,eq19,eq20,eq27}. We do not require to determine the full regions of validity $Z_{j}$ of the error term bounds \cref{eq24,eq25,eq26}. All we need are the following subsets of the regions of validity, which are straightforward to verify from the integral representation \cref{eq131} of $\xi$. Firstly, $Z_{1}\cap Z_{2}$ contains all points in the first quadrant $0 \leq \arg(z) \leq \frac12 \pi$, except for a small neighborhood of the interval $[0,a]$. Next, $Z_{1}\cap Z_{4}^{+}$ contains all points in the second quadrant $\frac12 \pi \leq \arg(z) \leq \pi$, except for a small neighborhood of the interval $[-a,0]$. Finally, $Z_{3}\cap Z_{2}$ and $Z_{3}\cap Z_{4}^{-}$ are conjugates of $Z_{1}\cap Z_{2}$ and $Z_{1}\cap Z_{4}^{+}$, respectively. Thus, from \cref{eq89} we see that the region $D$, which we shall use for our parabolic cylinder function approximations, consists of all points in the $z$ plane having cuts along $(-\infty,-1]$ and $[1,\infty)$ and with a small neighborhood of the interval $[-a,a]$ removed.

Next, to evaluate the coefficients in the LG expansions, we first find from \cref{eq04,eq129} that
\begin{equation}
\label{eq139}
\psi(a,z)=\frac{\left(1-z^2\right)\left\{(4a^2-3)z^2-a^2(2-a^2)\right\}}{4\left(z^2-a^2\right)^3}.
\end{equation}
Note that $\psi(a,z)=\mathcal{O}(1-z)$ as $z \to 1$, and also
\begin{equation}
\label{eq140}
\psi(a,z)=\frac{3-4a^2}{4z^2}+
\mathcal{O}\left(\frac{1}{z^4}\right)
\quad (z \to \infty),
\end{equation}
and both of these are consequences of $f(a,z)$ and $g(z)$, as defined by \cref{eq129}, satisfying the conditions at the beginning of \cref{sec:LG} at the singularities $z=\pm 1$ and $z=\infty$.

Now, let us evaluate the coefficients given by \cref{eq21,eq22,eq23}. In order to do so, similarly to (\ref{eq52}), let
\begin{equation}
\label{eq141}
\beta=\frac{z}{a^{2}\left(z^{2}-a^{2}\right)^{1/2}}-\frac{1}{a^{2}}
\quad (a>0),
\end{equation}
and
\begin{equation} 
\label{eq142}
\beta=\frac{1}{2z^{2}}
\quad (a=0).
\end{equation}
Then we denote $\hat{E}_{s}(a,z)=E_{s}(a,\beta)$ ($s=1,2,3,\ldots$), and so from \cref{eq21,eq22,eq23,eq139,eq141}, after some calculation, we obtain
\begin{equation} 
\label{eq143}
E_{1}(a,\beta)=\tfrac{1}{24}\beta
\left\{5a^{4}\left(1-a^{2}\right)\beta^{2}
+15a^{2}\left(1-a^{2}\right)\beta-12a^{2}+9\right\},
\end{equation}
\begin{multline} 
\label{eq144}
E_{2}(a,\beta)=\tfrac{1}{16}\beta\left(a^{2}\beta+2\right)
\left\{a^{2}\left(1-a^{2}\right)\beta^{2}
+2\left(1-a^{2}\right)\beta-1\right\}
\\ \times
\left\{5a^{2}\left(1-a^{2}\right)\beta^{2}
+10a^{2}\left(1-a^{2}\right)\beta-4a^{2}+3\right\},
\end{multline}
and for $s=2,3,4,\ldots$
\begin{equation} 
\label{eq145}
E_{s+1}(a,\beta) =
\frac{1}{2}G(a,\beta)
\frac{\partial E_{s}(a,\beta)}
{\partial \beta}
+\frac{1}{2}\int_{0}^{\beta}G(a,p)
\sum\limits_{j=1}^{s-1}
\frac{\partial E_{j}(a,p)}{\partial p}
\frac{\partial E_{s-j}(a,p)}{\partial p} dp,
\end{equation}
where
\begin{equation} 
\label{eq146}
G(a,\beta)=-d\beta/d \xi
=\beta\left(a^{2}\beta+2\right)
\left\{a^{2}\left(1-a^{2}\right)\beta^{2}+2\left(1-a^{2}\right)\beta-1\right\}.
\end{equation}

Note that our choice of integration constants imply that $E_{s}(a,0)=0$ ($s=1,2,3,\ldots$). In addition, the odd coefficients $E_{2s+1}(a,\beta)$ ($s=0,1,2,\ldots$), considered as functions of $z$, are analytic for all points bounded away from the interval $[-a,a]$, since $\beta=\beta(z)$ is analytic in this region and each $E_{s}(a,\beta)$ is a polynomial in $\beta$. Thus, from \cref{eq36,eq85}$, \hat{\alpha}=\alpha+\mathcal{O}(u^{-n})$ as $u \to \infty$ for any positive integer $n$. In fact, we shall show that $\hat{\alpha}=\alpha$ exactly. 

Next, we identify solutions of (\ref{eq111}) recessive at $z=\pm i \infty$ and $z=1$, with the corresponding LG expansions \cref{eq18,eq28,eq19,eq20,eq27}, which yields
\begin{equation} 
\label{eq147}
\mathsf{Q}^{\mu}_{\nu,1}(z)
=\sqrt{\frac{\pi}{2}}\frac{e^{-\frac12 \mu\pi i}u^{u}}
{\left(u^{2}-\mu^{2}\right)^{\frac12 u}\Gamma(u+1)}
\left(\frac{u-\mu}{u+\mu}\right)^{\frac12\mu}
\frac{W_{n,1}(u,a,z)}{\left(z^{2}-a^{2}\right)^{1/4}},
\end{equation}
\begin{equation} 
\label{eq148}
\mathsf{Q}^{\mu}_{\nu,-1}(z)
=\sqrt{\frac{\pi}{2}}\frac{e^{\frac12 \mu\pi i}u^{u}}
{\left(u^{2}-\mu^{2}\right)^{\frac12 u}\Gamma(u+1)}
\left(\frac{u-\mu}{u+\mu}\right)^{\frac12\mu}
\frac{W_{n,3}(u,a,z)}{\left(z^{2}-a^{2}\right)^{1/4}},
\end{equation}
and
\begin{multline} 
\label{eq149}
\mathsf{P}^{-\mu}_{\nu}(z)
=\frac{\left(u^{2}-\mu^{2}\right)^{\frac12 u}\,
\Gamma(u)}{\sqrt{2\pi}\,u^u\Gamma\left(u+\mu+\tfrac12\right)}
\left(\frac{u+\mu}{u-\mu}\right)^{\frac12 \mu}
\frac{W_{n,2}(u,a,z)}{\left(z^{2}-a^{2}\right)^{1/4}}
\\ \times
\left\{1+\eta_{n,2}(u,a,\infty) \right\}.
\end{multline}
In these, we determined the proportionality constants from \cref{eq114,eq121,eq122,eq126,eq128,eq138}. Note that, similarly to the LG functions of \cref{sec:PCF}, $W_{n,1}(u,a,z)$ and $W_{n,3}(u,a,z)$ are actually independent of $n$, but for consistency of notation, we keep this in their subscripts.

Alternatively, for (\ref{eq149}) we can find the constant as $z \to 1$ using \cref{eq120,eq136} in place of \cref{eq126,eq138} to get
\begin{multline} 
\label{eq150}
\mathsf{P}^{-\mu}_{\nu}(z)
=\left(\frac{\mu}{u}\right)^{1/4}
\frac{\mu^{\mu}}{\left(u^{2}-\mu^{2}\right)^{\frac12 \mu}\Gamma(\mu+1)}
\left(\frac{u-\mu}{u+\mu}\right)^{\frac12 u}
\frac{W_{n,2}(u,a,z)}{\left(z^{2}-a^{2}\right)^{1/4}}
\\ \times
\exp \left\{\sum\limits_{s=1}^{n-1}{(-1)^{s-1}
\frac{E_{s}(a,\beta_{1})}{u^{s}}}
\right\},
\end{multline}
where from \cref{eq128,eq141}
\begin{equation} 
\label{eq151}
\beta_{1}=\beta(\pm 1)
=\frac{u^{2}}{\mu(u+\mu)}.
\end{equation}

Similarly,
\begin{multline} 
\label{eq152}
\mathsf{P}^{-\mu}_{\nu}(-z)
=\frac{e^{\mp (\nu -\mu) \pi i}
\left(u^{2}-\mu^{2}\right)^{\frac12 u}\,
\Gamma(u)}{\sqrt{2\pi}\,u^u\Gamma\left(u+\mu+\tfrac12\right)}
\left(\frac{u+\mu}{u-\mu}\right)^{\frac12 \mu}
\frac{W_{n,4}^{\pm}(u,a,z)}{\left(z^{2}-a^{2}\right)^{1/4}}
\\ \times
\left\{1+\eta_{n,4}^{\pm}(u,a,\pm i\infty) \right\},
\end{multline}
or, alternatively,
\begin{multline} 
\label{eq153}
\mathsf{P}^{-\mu}_{\nu}(-z)
=e^{\mp \frac12 (u\alpha^2-1) \pi i}\left(\frac{\mu}{u}\right)^{1/4}
\frac{\mu^{\mu}}{\left(u^{2}-\mu^{2}\right)^{\frac12 \mu}\Gamma(\mu+1)}
\left(\frac{u-\mu}{u+\mu}\right)^{\frac12 u}
\\ \times
\frac{W_{n,4}^{\pm}(u,a,z)}{\left(z^{2}-a^{2}\right)^{1/4}}
\exp \left\{\sum\limits_{s=1}^{n-1}{(-1)^{s}
\frac{E_{s}(a,\beta_{1})}{u^{s}}}
\right\},
\end{multline}
noting that $(z^{2}-a^{2})^{1/4} \to e^{\pm \pi i/2}(1-a^{2})^{1/4}$ as $z \to -1 \pm i0$ with our choice of branches.

On comparing \cref{eq18,eq28,eq19,eq20,eq27} with \cref{eq147,eq148,eq149,eq152} we can then identify
\begin{equation}
\label{eq154}
w_{j}(u,a,z)=\left(1-z^{2}\right)^{1/2}e^{\pm \frac{1}{2}(\mu+1)\pi i}
\mathsf{Q}^{\mu}_{\nu,\pm 1}(z),
\end{equation}
for $j=1,3$, respectively, and 
\begin{equation}
\label{eq155}
w_{j}(u,a,z)=\frac{\pi}
{\Gamma(\nu-\mu+1)}
\left(1-z^{2}\right)^{1/2}
\mathsf{P}^{-\mu}_{\nu}(\pm z),
\end{equation}
for $j=2,4$, respectively. Furthermore, from \cref{eq147,eq148,eq154,eq155}
\begin{equation}
\label{eq156}
k_{n,j}(u,a)=\pm i\sqrt{\frac{\pi}{2}}
\frac{u^{u}}{\left(u^{2}-\mu^{2}\right)^{\frac12 u}
\Gamma(u+1)}\left(\frac{u-\mu}{u+\mu}\right)^{\frac12\mu},
\end{equation}
for $j=1,3$, respectively, these being independent of $n$; as above for $W_{n,1}(u,a,z)$ and $W_{n,3}(u,a,z)$, we keep this in the subscripts. 

Similarly, from \cref{eq149,eq152,eq154,eq155},
\begin{multline}
\label{eq157}
k_{n,2}(u,a)=\sqrt{\frac{\pi}{2}}
\frac{\left(u^{2}-\mu^{2}\right)^{\frac12 u}\,
\Gamma(u)}{u^u\Gamma\left(u+\mu+\tfrac12\right)
\Gamma\left(u-\mu+\tfrac12\right)}
\left(\frac{u+\mu}{u-\mu}\right)^{\frac12 \mu}
\\ \times
\left\{1+\eta_{n,2}(u,a,\infty) \right\},
\end{multline}
and 
\begin{equation}
\label{eq158}
k_{n,4}^{\pm}(u,a)=e^{\mp (u-\mu-\frac12)\pi i}k_{n,2}(u,a)
\left\{
\frac{1+\eta_{n,4}^{\pm}(u,a,\pm i\infty)}
{1+\eta_{n,2}(u,a,\infty)}
\right\}.
\end{equation}

Now let us construct our main results, namely asymptotic expansions for the Legendre functions which are valid in a domain containing the turning points. Firstly, from (\ref{eq11}), we have
\begin{equation} 
\label{eq159}
\alpha = \left\{\frac{2}{\pi}
\int_{-a}^{a} \frac{\left(a^2-t^2\right)^{1/2}}
{1-t^2} dt\right\}^{1/2}
=\left[ 2\left\{1-\left(1-a^{2}\right)^{1/2}\right\} \right]^{1/2},
\end{equation}
i.e.
\begin{equation} 
\label{eq160}
\tfrac12 \alpha^2
=1-\left(1-a^{2}\right)^{1/2},
\end{equation}
which from (\ref{eq128}) yields
\begin{equation} 
\label{eq161}
\tfrac12u\alpha^{2}=\nu-\mu+\tfrac12.
\end{equation}
With this value, $\zeta=\zeta(z)$ is defined by (\ref{eq08}), where $\xi=\xi(z)$ is given by (\ref{eq131}) and $\alpha$ by \cref{eq161}. Note that when $z=x \in (-a,a)$ ($a>0$) we have $\zeta \in (-\alpha,\alpha)$, such that
\begin{multline} 
\label{eq162}
\arcsin\left(\frac{x}{a}\right)
-\left(1-a^{2}\right)^{1/2}\arctan
\left\{x\left(\frac{1-a^{2}}{a^{2}-x^{2}}\right)^{1/2}\right\}
\\
=\frac{1}{2}\zeta\left(\alpha^{2}-\zeta^{2}\right)^{1/2}
+\frac{1}{2}\alpha^{2}\arcsin\left(\frac{\zeta}{\alpha}\right),
\end{multline}
on using (\ref{eq160}). Next, $\zeta \in [\alpha,\infty)$ when $z=x \in [a,1)$, and for $a>0$ is given implicitly by
\begin{multline}
\label{eq163}
\left(1-a^{2}\right)^{1/2}\mathrm{arctanh}
\left\{\frac{1}{x}\left(\frac{x^{2}-a^{2}}
{1-a^{2}}\right)^{1/2}\right\}
-\mathrm{arccosh}\left(\frac{x}{a}\right)
\\
=\frac{1}{2}\zeta\left(\zeta^{2}-\alpha^{2}\right)^{1/2}
-\frac{1}{2}\alpha^{2}\mathrm{arccosh}\left(\frac{\zeta}{\alpha}\right).
\end{multline}
Finally, $\zeta=\left\{-\ln\left(1-x^{2}\right)\right\}^{1/2}$ for 
$x \in (-1,1)$ when $a=0$.

If we compare \cref{eq158} to \cref{eq33,eq34} we infer that $\lambda \approx u-\mu-\frac12 = \nu - \mu$ as $u \to \infty$. In fact, from \cref{eq30,eq125,eq154,eq155} we see that $\lambda=\nu-\mu$ exactly, and hence from \cref{eq85,eq161}
\begin{equation}
\label{eq:alpha}
\tfrac12 u \hat{\alpha}^{2}
=\nu-\mu+\tfrac12
= \tfrac12 u \alpha^{2}
\implies \hat{\alpha} = \alpha.
\end{equation}

Next, using \cref{eq82,eq83,eq161,eq154,eq:alpha}, define $\mathcal{A}(u,a,z)$ and $\mathcal{B}(u,a,z)$ by the pair of equations
\begin{multline}
\label{eq164}
\mathsf{Q}^{\mu}_{\nu,\pm 1}(z)
=e^{\mp \frac12(\nu+1)\pi i}\Biggl\{
U\left(\nu-\mu+\tfrac{1}{2},\mp i\sqrt{2u}\,\zeta\right) 
\mathcal{A}(u,a,z)  \Biggr.
\\ \left.
+\frac{\partial U\left(\nu-\mu+\tfrac{1}{2},\mp i\sqrt{2u}
\,\zeta\right)}{\partial\zeta}\mathcal{B}(u,a,z)\right\}.
\end{multline}
Then from \cref{eq45,eq123,eq125}
\begin{multline}
\label{eq165}
\mathsf{P}^{-\mu}_{\nu}(\pm z)=
\sqrt{\frac{2}{\pi}}
\Biggl\{U\left(\mu-\nu-\tfrac{1}{2},
\pm \sqrt{2u}\,\zeta\right)\mathcal{A}(u,a,z) \Biggr.
\\ \left.
+\frac{\partial U\left(\mu-\nu-\tfrac{1}{2},\pm \sqrt{2u}\,
\zeta\right)}{\partial\zeta}\mathcal{B}(u,a,z)\right\}.
\end{multline}
Also from \cref{eq48,eq124,eq164}
\begin{multline}
\label{eq166}
\mathsf{Q}^{-\mu}_{\nu}(z)=
\sqrt{\frac{\pi}{2}}\, \Gamma(\nu-\mu+1)
\Biggl\{V\left(\mu-\nu-\tfrac{1}{2},
\sqrt{2u}\,\zeta\right)\mathcal{A}(u,a,z) \Biggr.
\\ \left.
+\frac{\partial V\left(\mu-\nu-\tfrac{1}{2},\sqrt{2u}\,
\zeta\right)}{\partial\zeta}\mathcal{B}(u,a,z)\right\}.
\end{multline}

On employing the parabolic cylinder function Wronskian relations \cref{eq43,eq44} we obtain, from solving the pair of equations \cref{eq164,eq165}, our desired explicit formulas
\begin{multline}
\label{eq167}
\mathcal{A}(u,a,z)=-\frac{e^{\frac12\mu\pi i}}{2\sqrt{u}}\left\{
e^{-\frac12(\nu-1)\pi i}\sqrt{\pi}\,\mathsf{P}^{-\mu}_{\nu}(\pm z)
\frac{\partial U\left(\nu-\mu+\tfrac12,-i\sqrt{2u}
\,\zeta\right)}{\partial\zeta} \right.
\\ \left. +
\sqrt{2}\,\mathsf{Q}^{\mu}_{\nu,1}(z)\frac{\partial 
U\left(\mu-\nu-\tfrac12,\pm \sqrt{2u}
\, \zeta\right)}{\partial\zeta}\right\},
\end{multline}
and
\begin{multline}
\label{eq168}
\mathcal{B}(u,a,z)=\frac{e^{\frac12 \mu\pi i}}
{2\sqrt{u}}\left\{
e^{-\frac12(\nu-1)\pi i}\sqrt{\pi}\,
\mathsf{P}^{-\mu}_{\nu}(\pm z)
U\left(\nu-\mu+\tfrac12,-i\sqrt{2u}
\,\zeta\right) \right.
\\ \left. +
\sqrt{2}\,\mathsf{Q}^{\mu}_{\nu,1}(z)
U\left(\mu-\nu-\tfrac12,\pm \sqrt{2u}\,\zeta\right)\right\}.
\end{multline}
Alternatively, on solving both of (\ref{eq165}),
\begin{multline}
\label{eq169}
\mathcal{A}(u,a,z)=\frac{\sqrt{2}\,\Gamma(\mu-\nu)}
{4\sqrt{u}}\left\{\mathsf{P}^{-\mu}_{\nu}(z)
\frac{\partial U\left(\mu-\nu-\tfrac12,
-\sqrt{2u}\,\zeta\right)}{\partial\zeta} \right.
\\ \left. -\mathsf{P}^{-\mu}_{\nu}(-z)
\frac{\partial U\left(\mu-\nu-\tfrac12,\sqrt{2u}
\,\zeta\right)}{\partial\zeta}\right\},
\end{multline}
and
\begin{multline}
\label{eq170}
\mathcal{B}(u,a,z)=-\frac{\sqrt{2}\,\Gamma(\mu-\nu)}{4\sqrt{u}}
\left\{\mathsf{P}^{-\mu}_{\nu}(z)
U\left(\mu-\nu-\tfrac12,-\sqrt{2u}
\,\zeta\right) \right.
\\ \left. -\mathsf{P}^{-\mu}_{\nu}(-z) 
U\left(\mu-\nu-\tfrac12,\sqrt{2u}
\,\zeta\right)\right\},
\end{multline}
which demonstrates that $\mathcal{A}(u,a,z)$ and $\mathcal{B}(u,a,z)$ are real for $z=x \in (-1,1)$, and analytic in the principal $z$ plane having cuts along $(-\infty,-1]$ and $[1,\infty)$.

Next, in a general setting, the asymptotic expansions for $\mathcal{A}(u,a,z)$ and $\mathcal{B}(u,a,z)$ come from \cref{eq90,eq91}. In this specific case, we can insert \cref{eq59,eq60,eq66,eq68,eq78,eq80,eq75,eq76,eq147,eq149} into \cref{eq167,eq168} (with upper signs taken), to arrive at
\begin{multline}
\label{eq171}
\mathcal{A}(u,a,z)=\left(\frac{\zeta^{2}
-\alpha^{2}}{z^{2}-a^{2}}\right)^{1/4}
\left[C_{1}(u,\alpha)\exp\left\{\sum\limits_{s=1}^{n-1}
\frac{\tilde{\mathcal{E}}_{s}(a,z)}{u^{s}}\right\}
\left\{1+\tilde{\delta}_{n,1}(u,a,z)\right\}
\right.
\\
\left.
+C_{2}(u,\alpha)\exp\left\{\sum\limits_{s=1}^{n-1}(-1)^{s}
\frac{\tilde{\mathcal{E}}_{s}(a,z)}{u^{s}}\right\}
\left\{1+\tilde{\delta}_{n,2}(u,a,z)\right\}
\right],
\end{multline}
and
\begin{multline}
\label{eq172}
\mathcal{B}(u,a,z)=\frac{1}
{u\left\{\left(z^{2}-a^{2}\right)
\left(\zeta^{2}-\alpha^{2}\right)\right\}^{1/4}}
\left[C_{1}(u,\alpha)\exp\left\{\sum\limits_{s=1}^{n-1}
\frac{\mathcal{E}_{s}(a,z)}{u^{s}}\right\}
\right.
\\
\left.
\times \left\{1+\delta_{n,1}(u,a,z)\right\}
-C_{2}(u,\alpha)\exp\left\{\sum
\limits_{s=1}^{n-1}(-1)^{s}
\frac{\mathcal{E}_{s}(a,z)}{u^{s}}\right\}
\left\{1+\delta_{n,2}(u,a,z)\right\}
\right],
\end{multline}
where
\begin{equation} 
\label{eq173}
C_{1}(u,\alpha)=\frac{2^{\frac{3}{4}}
\sqrt{\pi}\,e^{\frac{1}{2}(\mu-u)}
u^{u-\frac{3}{4}}(u+\mu)^{-\frac{1}{2}(u+\mu)}}
{4\,\Gamma(u)},
\end{equation}
\begin{equation} 
\label{eq174}
C_{2}(u,\alpha)=\frac{2^{\frac{1}{4}}
e^{\frac{1}{2}(u-\mu)}
u^{\frac{1}{4}-u}(u+\mu)^{\frac{1}{2}
(u+\mu)}\Gamma(u)}
{4\,\Gamma\left(u+\mu+\tfrac12\right)},
\end{equation}
\begin{equation} 
\label{eq175}
\tilde{\mathcal{E}}_{s}(a,z)
=E_{s}(a,\beta)+(-1)^{s}\tilde{e}_{s}
(\alpha,\hat{\beta}),
\end{equation}
\begin{equation} 
\label{eq176}
\mathcal{E}_{s}(a,z)
=E_{s}(a,\beta)+(-1)^{s}e_{s}
(\alpha,\hat{\beta}),
\end{equation}
\begin{equation} 
\label{eq177}
\tilde{\delta}_{n,1}(u,a,z)
=\left\{1+\eta_{n,1}(u,a,z) \right\}
\left\{1+\tilde{\eta}_{n,2}
(u,\alpha,\zeta) \right\}-1,
\end{equation}
\begin{equation} 
\label{eq178}
\tilde{\delta}_{n,2}(u,a,z)
=\left\{1+\tilde{\eta}_{n,1}
(u,\alpha,\zeta) \right\}
\left\{1+\eta_{n,2}(u,a,z) \right\}
\left\{1+\eta_{n,2}(u,a,\infty) \right\}-1,
\end{equation}
\begin{equation} 
\label{eq179}
\delta_{n,1}(u,a,z)
=\left\{1+\eta_{n,1}(u,a,z) \right\}
\left\{1+\hat{\eta}_{n,2}(u,\alpha,\zeta) \right\}-1,
\end{equation}
and
\begin{equation} 
\label{eq180}
\delta_{n,2}(u,a,z)
=\left\{1+\hat{\eta}_{n,1}
(u,\alpha,\zeta) \right\}
\left\{1+\eta_{n,2}(u,a,z) \right\}
\left\{1+\eta_{n,2}(u,a,\infty) \right\}-1,
\end{equation}
on recalling that $\hat{\alpha}=\alpha$ (see \cref{eq:alpha}), and $\mu=(1-a^2)^{1/2}u=(1-\frac12 \alpha^2)u$ (see \cref{eq128,eq160}).

As we showed in the more general setting, by using other pairs of \cref{eq167,eq168,eq169,eq170}, these expansions remain valid in the full domain $D$ which includes the singularities $z=\pm 1$ and $z=\pm i \infty$, and surrounds, but does not include, the interval $[-a,a]$ containing the turning points. The expansions have the same form and differ only in their error terms; each is $\mathcal{O}(u^{-n})$ as $u \to \infty$ uniformly in their respective subdomains.

Next, we express \cref{eq171,eq172} in neater forms and now drop the error terms (which all have simple explicit error bounds). To this end, we have from \cref{eq173,eq174}
\begin{equation} 
\label{eq181}
\sqrt{C_{1}(u,\alpha)C_{2}(u,\alpha)}
=\left(\frac{\pi}{u}\right)^{1/4}
\frac{\sqrt{2}}{4\,\sqrt{\Gamma(u+\mu+\tfrac12)}}.
\end{equation}
Now, from \cite[Eqs. 5.5.5 and 5.11.1]{NIST:DLMF}
\begin{equation} 
\label{eq182}
\frac12 \ln \left\{\frac{C_{1}(u,\alpha)}
{C_{2}(u,\alpha)}\right\}
\sim \sum_{s=0}^{\infty}\frac{d_{2s+1}(\alpha)}{u^{2s+1}}
\quad (u \to \infty),
\end{equation}
noting that only odd powers of $u$ appear here. The first three terms are found to be
\begin{equation} 
\label{eq183}
d_{1}(\alpha)=-\frac{9-2\alpha^{2}}
{24\left(4-\alpha^{2}\right)},
\end{equation}
\begin{equation} 
\label{eq184}
d_{3}(\alpha)=\frac{135- 96\alpha^2+ 24\alpha^4-2\alpha^6}
{720\left(4-\alpha^{2}\right)^{3}},
\end{equation}
and
\begin{equation} 
\label{eq185}
d_{5}(\alpha)=-\frac{2079 - 2560\alpha^2 + 1280\alpha^4 
- 320\alpha^6 + 40\alpha^8 - 2\alpha^{10}}
{2520\left(4-\alpha^{2}\right)^{5}}.
\end{equation}
Thus, using the elementary identity
\begin{equation} 
\label{eq186}
c_{1}w_{1}+c_{2}w_{2}=\sqrt{c_{1}c_{2}}
\left[\exp\left\{\tfrac12
\ln\left(c_{1}/c_{2}\right)\right\}w_{1}
+\exp\left\{-\tfrac12
\ln\left(c_{1}/c_{2}\right)\right\}w_{2}\right],
\end{equation}
we derive from \cref{eq171,eq172,eq181,eq182}
\begin{equation} 
\label{eq187}
\mathcal{A}(u,a,z) =
\frac{\pi^{1/4}}{u^{1/4}\sqrt{2 \Gamma(u+\mu+\tfrac12)}}
\left(\frac{\zeta^{2}-\alpha^{2}}{z^{2}-a^{2}}\right)^{1/4}
A(u,a,z),
\end{equation}
where
\begin{equation}
\label{eq188}
A(u,a,z) \sim 
\exp\left\{\sum\limits_{s=1}^{\infty}
\frac{\tilde{\mathcal{E}}_{2s}(a,z)}{u^{2s}}\right\}
\cosh\left\{\sum\limits_{s=0}^{\infty}
\frac{\tilde{\mathcal{E}}_{2s+1}(a,z)+d_{2s+1}(\alpha)}{u^{2s+1}}\right\},
\end{equation}
and 
\begin{equation} 
\label{eq189}
\mathcal{B}(u,a,z) =
\frac{\pi^{1/4}}{u^{1/4}\sqrt{2 \Gamma(u+\mu+\tfrac12)}}
\left(\frac{\zeta^{2}-\alpha^{2}}{z^{2}-a^{2}}\right)^{1/4}
B(u,a,z),
\end{equation}
where
\begin{multline} 
\label{eq190}
B(u,a,z) \sim \frac{1}{u \left(\zeta^{2}
-\alpha^{2}\right)^{1/2}}
\exp\left\{\sum\limits_{s=1}^{\infty}
\frac{\mathcal{E}_{2s}(a,z)}{u^{2s}}\right\}
\\ \times
\sinh\left\{\sum\limits_{s=0}^{\infty}
\frac{\mathcal{E}_{2s+1}(a,z)+d_{2s+1}(\alpha)}{u^{2s+1}}\right\}.
\end{multline}
These expansions are uniformly valid as $u \to \infty$ for $z$ lying in the principal complex plane having cuts along $(-\infty,-1]$ and $[1,\infty)$, with points on and close to $[-a,a]$ removed. For these excluded points, we can either use the Cauchy integral method \cref{eq97}, or in this case a re-expansion in inverse powers of $u^{2}$, as described next.

It can be easily verified that the conditions of \cref{thm:tp-series} are met here, and only even inverse powers of $u$ appear in the asymptotic series. Thus, the following expansions hold
\begin{equation} 
\label{eq191}
A(u,a,z) \sim 1 + \sum_{s=1}^{\infty}
\frac{\mathrm{A}_{2s}(a,z)}{u^{2s}},
\quad
B(u,a,z) \sim \frac{1}{u^2}\sum_{s=0}^{\infty}
\frac{\mathrm{B}_{2s}(a,z)}{u^{2s}},
\end{equation}
as $u \to \infty$, uniformly $0 \leq a \leq 1-\delta <1$ and for $z$ lying in the principal complex plane having cuts along $(-\infty,-1]$ and $[1,\infty)$. Here, $\mathrm{A}_{2s}(a,z)$ ($s=1,2,3,\ldots$) and $\mathrm{B}_{2s}(a,z)$ ($s=0,1,2,\ldots$) are analytic at all points in the cut plane, including at $z= \pm a$. These coefficients can be explicitly determined by formally expanding \cref{eq188,eq190} in inverse powers of $u$. For example, the leading terms of both sets of coefficients are given in \cref{eq193,eq194}. They are also analytic at $z=0$, and using (\ref{eq162}) we have for $|z|<1$ and $a>0$ that
\begin{multline}
\label{eq192}
\zeta=\frac{a}{\alpha}z-\frac{\alpha
\left(7\alpha^{2}-24\right)}{96\,a}z^{3}
+\frac{\alpha^{3}\left(317\alpha^4 - 2320\alpha^2 + 4160\right)}{30720\,a^{3}}z^{5}
\\
-\frac{\alpha^{5}\left(36683\alpha^6 
- 413224\alpha^4 + 1537088\alpha^2 
- 1881600\right)}{20643840\,a^{5}}z^{7}+\ldots,
\end{multline}
and hence from \cref{eq188,eq190}
\begin{multline}
\label{eq193}
\mathrm{A}_{2}(a,z)=\frac12
\left\{\tilde{\mathcal{E}}_{1}(a,z)+d_{1}(\alpha)\right\}^{2}
+\tilde{\mathcal{E}}_{2}(a,z)
\\
=-\frac{1}{8\left(4-\alpha^{2}\right)^{2}}
-\frac{19\alpha^{2}}{128\left(4-\alpha^{2}\right)^{3}}z^{2}
-\frac{\alpha^4 + 608\alpha^2 - 304}{1536\left(4-\alpha^{2}\right)^{4}}z^{4}
+\ldots,
\end{multline}
and
\begin{multline}
\label{eq194}
\mathrm{B}_{0}(a,z)
=\frac{\mathcal{E}_{1}(a,z)+d_{1}(\alpha)}
{\left(\zeta^{2}-\alpha^{2}\right)^{1/2}}
\\
=\frac{\alpha^{3}}{64a^{3}}z+\frac{\alpha^{5}\left(\alpha^{2}+12\right)}{3072a^{5}}z^{3}
-\frac{\alpha^{7}\left(19\alpha^4 - 32\alpha^2 - 1216\right)}{655360a^{7}}z^{5}
+\ldots.
\end{multline}

Now, again from (\ref{eq162}) (or (\ref{eq163})), we obtain for $|z- a|<1-a$ ($a>0$)
\begin{multline}
\label{eq195}
\zeta=\alpha+\varpi (z-a)
-\frac{\varpi}{10}\left\{\frac{\varpi}{\alpha}
-\frac{7a^{2}+1}{a\left(1-a^2\right)}\right\}(z-a)^{2}
\\
+\frac{\varpi}{25}\left\{\frac{11\varpi^{2}}
{14\alpha^{2}}
-\frac{\left(7a^{2}+1\right)\varpi}
{2a\left(1-a^{2}\right)\alpha}
+\frac{2\left(51a^4 + 36a^2 - 1\right)}{7a^{2}
\left(1-a^{2}\right)^{2}}\right\}(z-a)^{3}
\\
-\frac{\varpi}{1750}\left\{\frac{823\varpi^{3}}
{36\alpha^{3}}
-\frac{33\left(7a^{2}+1\right)\varpi^{2}}
{2a\left(1-a^{2}\right)\alpha^{2}}
+\frac{\left(1159a^4 + 674a^2 - 9\right)\varpi}
{4a^{2}\left(1-a^{2}\right)^{2}\alpha}
\right.
\\
\left.
-\frac{8141a^6 + 13639a^4 + 327a^2 + 37}
{9a^{3}\left(1-a^{2}\right)^{3}}\right\}(z-a)^{4}
+ \ldots,
\end{multline}
where
\begin{equation} 
\label{eq196}
\varpi=\varpi(a)=\left\{\frac{a}
{\left(1-a^{2}\right)^{2}\alpha}\right\}^{1/3}.
\end{equation}
Both \cref{eq192,eq195} hold in the limit as $a \to 0$ ($\alpha \to 0$), and for both one can show in this case, on referring to \cref{eq160,eq196}, that $\zeta=z+\tfrac14 z^{3}+\tfrac{13}{96}z^{5}+\cdots$.

The construction of the corresponding Taylor series at $z=a$ for $\mathrm{A}_{2s}(a,z)$ ($s=1,2,3,\ldots$) and $\mathrm{B}_{2s}(a,z)$ ($s=0,1,2,\ldots$) becomes extremely unwieldy, for even a small number of terms, if $a$ is kept as a general parameter. Instead, for each prescribed value of $a$, it is more practical to numerically compute the series and then discard singular terms (fractional and negative powers of $(z-a$)), since we know that they vanish, even though numerically they are not identically zero due to rounding errors.

For example, for $a=0.5$ we find numerically in Maple (with Digits set to 40) that the Taylor series for $\mathrm{A}_{2}(0.5,z)$ is given by
\begin{multline} 
\label{eq197}
\frac{8.8\times10^{-41}}{(z-0.5)^3}
+\frac{1.4\times10^{-40}}{(z-0.5)^{5/2}}
-\frac{1.2\times10^{-39}}{(z-0.5)^{2}}
+\frac{4.8\times10^{-40}}{(z-0.5)^{3/2}}
\\
-\frac{2.2\times10^{-39}}{z-0.5}
-\frac{2.3\times10^{-39}}{(z-0.5)^{1/2}}
-0.0091223906
+1.1\times10^{-38}(z-0.5)^{1/2}
\\
-0.00034392390(z-0.5)
+2.7\times10^{-38}(z-0.5)^{3/2}
+0.0010293232(z-0.5)^{2}+\ldots,
\end{multline}
where here we have recorded the singular terms to two significant figures, and the nonsingular ones to 8 digits. Throwing out the superfluous terms yields the approximation that is appropriate for $z$ close to the turning point
\begin{multline}
\label{eq198}
\mathrm{A}_{2}(0.5,z) \approx -0.0091223906
-0.00034392390(z-0.5)
\\
+0.0010293232(z-0.5)^{2}+\ldots.
\end{multline}
Subsequently, we actually shall employ more digits for these coefficients.

We summarize the main results of this section. The Legendre functions have representations \cref{eq164,eq165,eq166}, where $\zeta=\zeta(z)$ is defined by (\ref{eq08}), $\xi=\xi(z)$ is given by \cref{eq131,eq132}, and $\alpha$ by \cref{eq161}. Here, the Legendre functions $\mathsf{Q}^{\mu}_{\nu,\pm 1}(z)$ are defined by \cref{eq121,eq122}, and $\mathsf{P}^{-\mu}_{\nu}(z)$ and $\mathsf{Q}^{-\mu}_{\nu}(z)$ are the analytic continuations into the complex plane of the standard Ferrers functions, and are expressible in the forms \cref{eq123,eq124}.

For $z$ lying in the principal complex plane having cuts along $(-\infty,-1]$ and $[1,\infty)$, with points on and close to the interval $[-a,a]$ removed, the coefficient functions $\mathcal{A}(u,a,z)$ and $\mathcal{B}(u,a,z)$ possess the expansions \cref{eq187,eq188,eq189,eq190} as $u=\nu+\frac12 \to \infty$, uniformly for $\delta u \leq \mu \leq u$. In these, the coefficients $\tilde{\mathcal{E}}_{s}(a,z)$ and $\mathcal{E}_{s}(a,z)$ are polynomials in $\beta$ and $\hat{\beta}$, as given by \cref{eq52,eq53,eq56,eq57,eq58,eq73,eq74,eq141,eq142,eq143,eq144,eq145,eq146,eq175,eq176}, in which $\hat{\alpha}=\alpha$. The constants $d_{2s+1}(\alpha)$ can be determined via \cref{eq173,eq174,eq182} and the asymptotic expansion of the gamma function \cite[Eq. 5.11.1]{NIST:DLMF}.

For points near or in the interval $[-a,a]$ the expansions \cref{eq187,eq189,eq191} are valid, where $\mathrm{A}_{2s}(a,z)$ and $\mathrm{B}_{2s}(a,z)$ are analytic at $z=\pm a$, and can be evaluated by expanding \cref{eq188,eq190} in inverse powers of $u$.

\section{Numerical results}
\label{sec:numerics}

If $\mathsf{P}^{-\mu}_{\nu}(x)$ has one or more zeros in $[0,1)$ we define an envelope function for it by
\begin{equation}
\label{eq199}
M(\nu,\mu,x)=
    \begin{cases}
        \left[\left\{\mathsf{P}^{-\mu}_{\nu}(x)\right\}^{2}
        +\left\{2\mathsf{Q}^{-\mu}_{\nu}(x)/\pi
        \right\}^{2}\right]^{1/2} 
        & (0<x \leq q^{-\mu}_{\nu,1})\\
        \mathsf{P}^{-\mu}_{\nu}(x) & (q^{-\mu}_{\nu,1} < x < 1)
    \end{cases},
\end{equation}
where $q^{-\mu}_{\nu,1}$ is the largest positive zero of $\mathsf{Q}^{-\mu}_{\nu}(x)$. In the oscillatory interval $(0,q^{-\mu}_{\nu,1})$ this positive function is close to the amplitude of $\mathsf{P}^{-\mu}_{\nu}(x)$, while coinciding with the function in the nonoscillatory interval. For example, in \cref{fig:EnvP} the function $\mathsf{P}^{-\mu}_{\nu}(x)$ (solid curve) and its envelope $M(\nu,\mu,x)$ (dotted curve), together with the reflection $-M(\nu,\mu,x)$ (dash-dotted curve) and $\mathsf{Q}^{-\mu}_{\nu}(x)$ (dashed curve), are shown for $\nu=50$ ($u=101/2$) and $\mu=(1-a^2)^{1/2}u$ ($u=\nu+\frac12$) with $a=0.5$ and $x \in [0,1)$. In the use of (\ref{eq199}) we find numerically for these parameter values that $q_{\nu,1}^{-\mu}=0.42542\cdots$.

\begin{figure}
 \centering
 \includegraphics[
 width=0.7\textwidth,keepaspectratio]{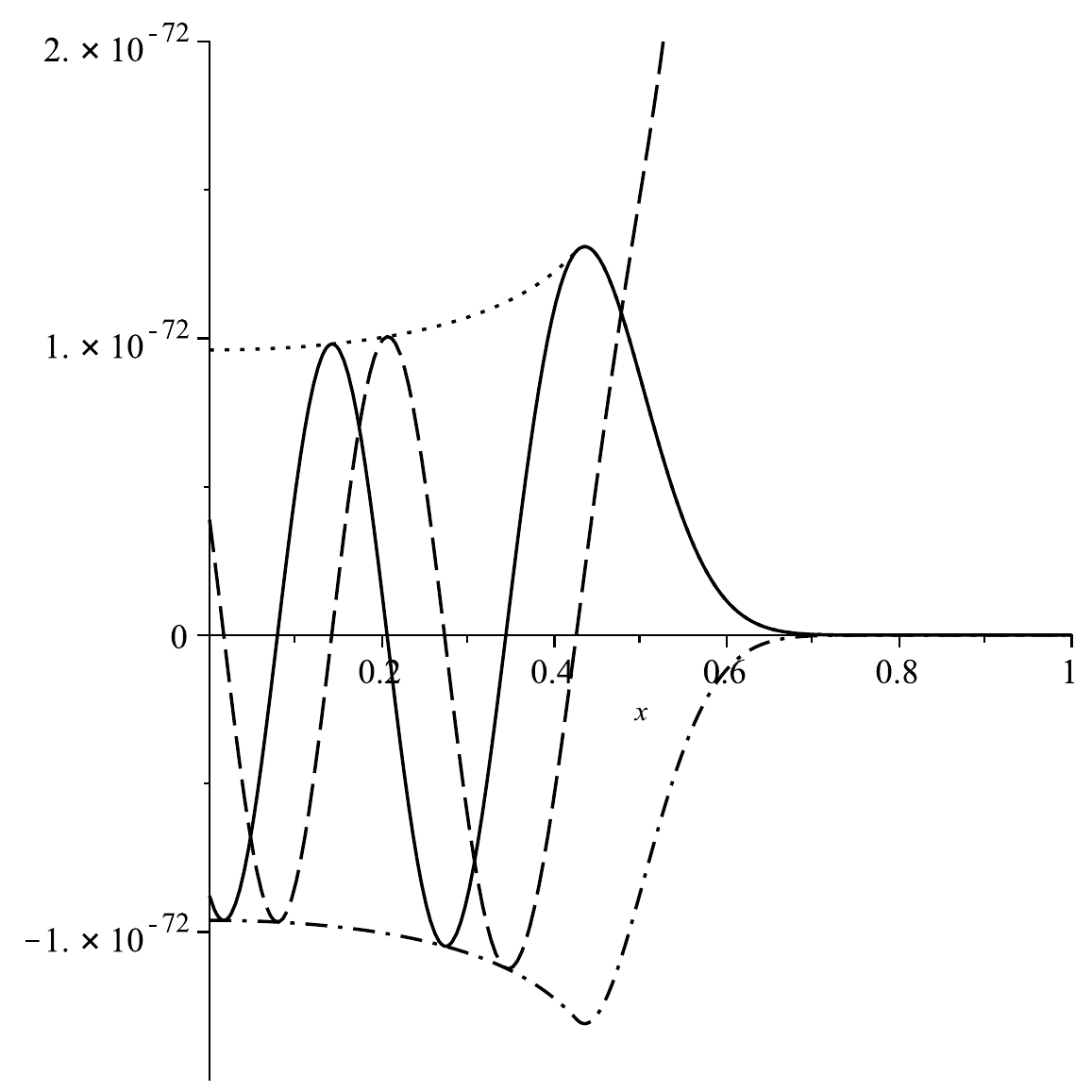}
 \caption{Graphs of $\mathsf{P}^{-\mu}_{\nu}(x)$ (solid curve), $2\mathsf{Q}^{-\mu}_{\nu}(x)/\pi$ (dashed curve), $M(\nu,\mu,x)$ (dotted curve) and $-M(\nu,\mu,x)$ (dash-dotted curve) for $\nu=50$ and $\mu=(1-a^2)^{1/2}u$ ($u=\nu+\frac12$) with $a=0.5$.}
 \label{fig:EnvP}
\end{figure}

\begin{figure}
 \centering
 \includegraphics[
 width=0.7\textwidth,keepaspectratio]{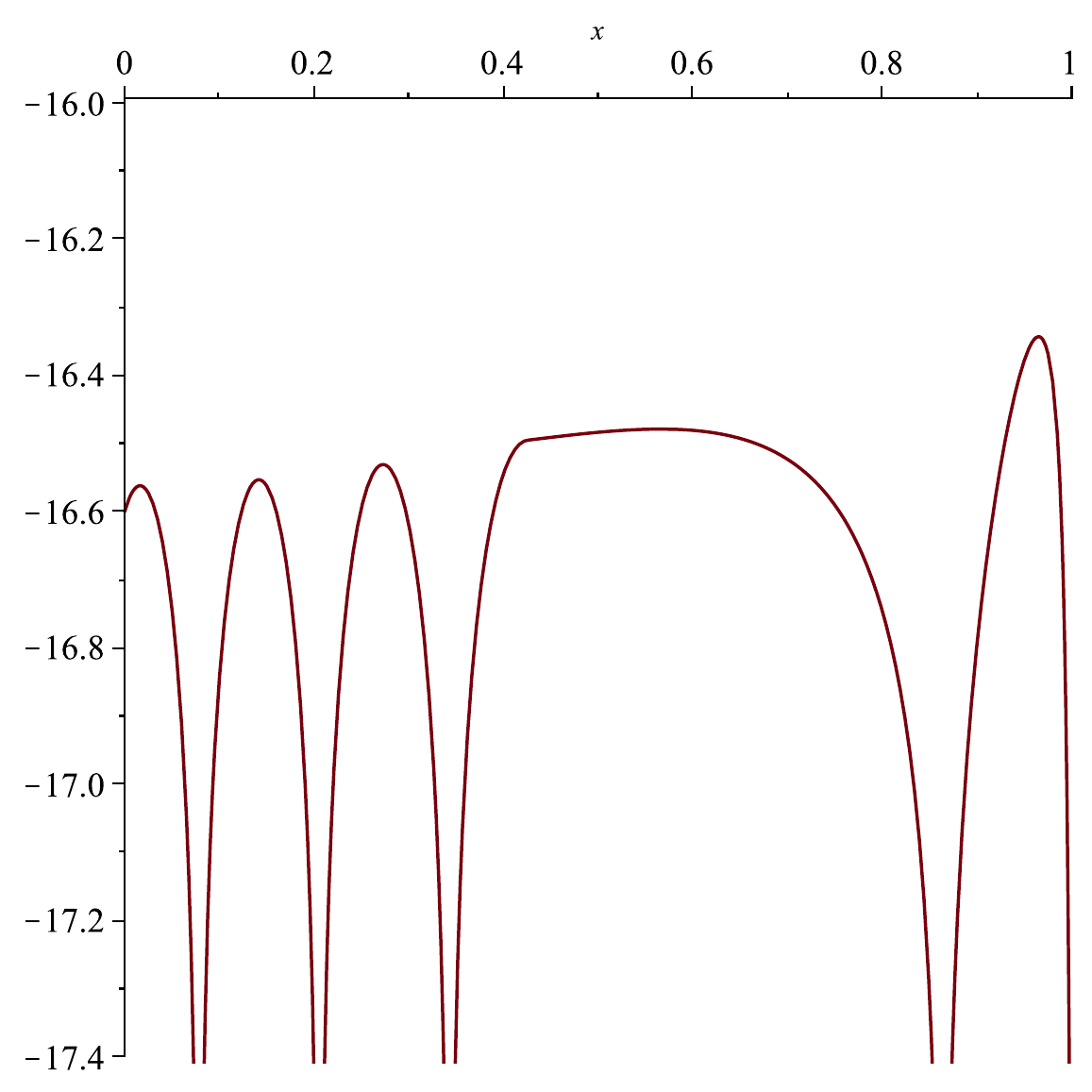}
 \caption{Graph of $\Omega_{4}(\nu,\mu,x)$ for $\nu=50$ and $\mu=(1-a^2)^{1/2}u$ ($u=\nu+\frac12$), with $a=0.5$.}
 \label{fig:Err1}
\end{figure}

\begin{figure}
 \centering
 \includegraphics[
 width=0.7\textwidth,keepaspectratio]{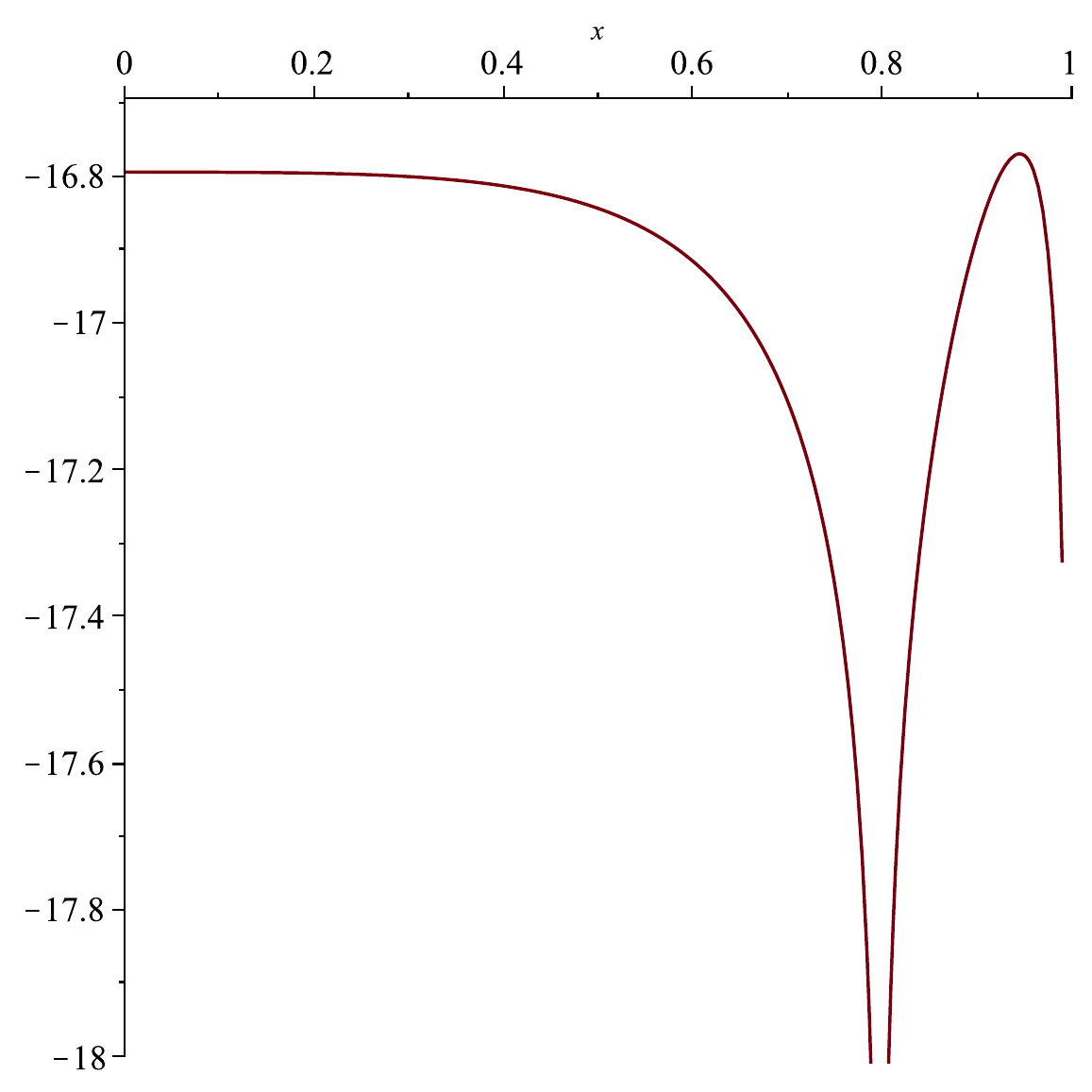}
 \caption{Graph of $\Omega_{4}(\nu,\mu,x)$ for $\nu=50$ and $\mu=(1-a^2)^{1/2}u$ ($u=\nu+\frac12$), with $a=0.1$.}
 \label{fig:Err2}
\end{figure}

Next, for $z=x \in [0,1)$, on taking the first four terms in \cref{eq187,eq188,eq189,eq190} define the approximations
\begin{equation} 
\label{eq200}
\mathcal{A}_{4}(u,a,x) =
\frac{\pi^{1/4}}{u^{1/4}\sqrt{2 \Gamma(u+\mu+\tfrac12)}}
\left(\frac{\zeta^{2}-\alpha^{2}}{x^{2}-a^{2}}\right)^{1/4}
\left\{1 + \sum_{s=1}^{3}
\frac{\mathrm{A}_{2s}(a,x)}{u^{2s}}\right\},
\end{equation}
and
\begin{equation} 
\label{eq201}
\mathcal{B}_{4}(u,a,x) =
\frac{\pi^{1/4}}{u^{9/4}\sqrt{2 \Gamma(u+\mu+\tfrac12)}}
\left(\frac{\zeta^{2}-\alpha^{2}}{x^{2}-a^{2}}\right)^{1/4}
\sum_{s=0}^{3}
\frac{\mathrm{B}_{2s}(a,x)}{u^{2s}},
\end{equation}
where $\zeta = \zeta(x) \in [0,\infty)$ is given by \cref{eq162,eq163}. Then from \cref{eq165,eq200,eq201} consider the following error for $z=x \in [0,1)$ in the approximation
\begin{multline}
\label{eq202}
\Delta_{4}(\nu,\mu,x)
=\mathsf{P}^{-\mu}_{\nu}(x)-
\sqrt{\frac{2}{\pi}}
\Biggl\{U\left(\mu-\nu-\tfrac{1}{2}, \sqrt{2u}\,\zeta\right)\mathcal{A}_{4}(u,a,x) \Biggr.
\\ \left.
+\frac{\partial U\left(\mu-\nu-\tfrac{1}{2},\sqrt{2u}\,
\zeta\right)}{\partial\zeta}\mathcal{B}_{4}(u,a,x)\right\}.
\end{multline}
To assess the relative accuracy of this, we computed the function
\begin{equation} 
\label{eq203}
\Omega_{4}(\nu,\mu,x)
=\log_{10}\left\{\frac{\left|\Delta_{4}(\nu,\mu,x)\right|}{M(\nu,\mu,x)}\right\}
\quad (0 \leq x <1).
\end{equation}

A plot of $\Omega_{4}(\nu,\mu,x)$ with $\nu=50$ ($u=101/2$) and $\mu=(1-a^2)^{1/2}u$ is given in \cref{fig:Err1} for $a=0.5$, and in \cref{fig:Err2} for $a=0.1$. In the latter case $\mathsf{P}^{-\mu}_{\nu}(z)$ does not have zeros in $[0,1)$ and so we took $M(\nu,\mu,x)=\mathsf{P}^{-\mu}_{\nu}(z)$ in (\ref{eq203}).

In both cases, we used the Taylor series about $x=a$ for $|x-a|<0.08$ as described above (see \cref{eq198}), and computed the coefficients in \cref{eq200,eq201} directly for all other values of $x$. The plots demonstrate about 16 digits or better of accuracy for both values of $a$, uniformly throughout the interval.

\section*{Acknowledgments}
I thank the anonymous referees for many helpful comments.

Financial support from Ministerio de Ciencia e Innovación pro\-ject PID2021-127252NB-I00 (MCIN/AEI/10.13039/ 501100011033/FEDER, UE) is acknowledged.

\section*{Conflict of interest}
The author declares no conflicts of interest.

\makeatletter
\interlinepenalty=10000

\bibliographystyle{siamplain}
\bibliography{biblio}
\end{document}